\documentclass[11pt,oneside,english]{article}

\usepackage[utf8]{inputenc}
\usepackage[a4paper,margin=1in]{geometry}
\usepackage{color}

\usepackage{amstext}
\usepackage{amsmath}
\usepackage{amsthm}

\usepackage{amssymb}
\usepackage{thmtools}
\usepackage{thm-restate}
\usepackage[autostyle]{csquotes}
\usepackage{nicefrac}
\usepackage{enumitem}
\usepackage{pifont}

\usepackage[
backend=biber,
style=alphabetic,
citestyle=alphabetic
]{biblatex}

\AtBeginDocument{\providecommand\lemref[1]{\cref{lem:#1}}}
\AtBeginDocument{}
\AtBeginDocument{\providecommand\defref[1]{\cref{def:#1}}}

\makeatletter

\newcommand*{\underarrow}{\def\@underarrow{\relax}\@ifstar{\@@underarrow}{\def\@underarrow{\hidewidth}\@@underarrow}}
\newcommand*{\@@underarrow}[2][]{\underset{\@underarrow\substack{\uparrow\if\relax\detokenize{#1}\relax\else\\#1\fi}\@underarrow}{#2}}

\newcommand*{\overarrow}{\def\@overarrow{\relax}\@ifstar{\@@overarrow}{\def\@overarrow{\hidewidth}\@@overarrow}}
\newcommand*{\@@overarrow}[2][]{\overset{\@overarrow\substack{\if\relax\detokenize{#1}\relax\else#1\\\fi\downarrow}\@overarrow}{#2}}
\makeatother

\makeatletter
\def\moverlay{\mathpalette\mov@rlay}
\def\mov@rlay#1#2{\leavevmode\vtop{
   \baselineskip\z@skip \lineskiplimit-\maxdimen
   \ialign{\hfil$\m@th#1##$\hfil\cr#2\crcr}}}
\newcommand{\charfusion}[3][\mathord]{
    #1{\ifx#1\mathop\vphantom{#2}\fi
        \mathpalette\mov@rlay{#2\cr#3}
      }
    \ifx#1\mathop\expandafter\displaylimits\fi}
\makeatother

\newcommand{\cupdot}{\charfusion[\mathbin]{\cup}{\cdot}}
\newcommand{\bigcupdot}{\charfusion[\mathop]{\bigcup}{\cdot}}

\theoremstyle{remark}
\newtheorem*{rem*}{Remark}
\theoremstyle{plain}
\newtheorem{thm}{\protect\theoremname}
\theoremstyle{definition}
\newtheorem{defn}{\protect\definitionname}
\theoremstyle{plain}
\newtheorem{lem}{\protect\lemmaname}
\theoremstyle{remark}
\newtheorem{claim}{Claim}
\theoremstyle{plain}
\newtheorem{fac}{Fact}
\theoremstyle{plain}
\newtheorem{cor}{Corollary}
\theoremstyle{plain}

\declaretheorem[name=Proposition,Refname={Proposition,Proposition}]{prop}

\theoremstyle{plain}

\usepackage{url}
\usepackage{relsize}

\usepackage{bbm}

\usepackage{babel}

\providecommand{\lemmaname}{Lemma}

\providecommand{\definitionname}{Definition}
\providecommand{\theoremname}{Theorem}

\usepackage{varioref}
\usepackage[unicode]{hyperref}

\usepackage{cleveref}

\crefname{defn}{definition}{Definition}
\Crefname{defn}{Definition}{Definitions}
\crefname{claim}{claim}{claims}
\Crefname{claim}{Claim}{Claims}
\crefname{prop}{proposition}{propositions}
\Crefname{prop}{Proposition}{Propositions}
\crefname{cor}{corollary}{corollaries}
\Crefname{cor}{Corollary}{Corollaries}
\crefname{thm}{theorem}{theorems}
\Crefname{thm}{Theorem}{Theorems}

\crefname{lem}{lemma}{lemmas}
\Crefname{lem}{Lemma}{Lemmas}
\DeclareMathOperator{\Inv}{Inv}

\DeclareMathOperator{\HPOWER}{HPOWER}

  \def\acts{\curvearrowright} 

\begin{document}

\title{Transitive bounded-degree 2-expanders from regular 2-expanders}

\author{Eyal Karni$^1$ \and Tali Kaufman$^2$}
\date{
    $^1$Bar Ilan University, ISRAEL. \\  eyalk5@gmail.com\\
    $^2$ Bar Ilan University, ISRAEL. \\  kaufmant@mit.edu\\[2ex]
    \today
}

\maketitle

\begin{abstract}

A two-dimensional simplicial complex is called $d$-{\em regular} if every edge of it is contained in exactly $d$ distinct triangles. It is called $\epsilon$-expanding if its up-down two-dimensional random walk has a normalized maximal eigenvalue which is at most $1-\epsilon$.
  In this work, we present a class of bounded degree 2-dimensional expanders, which is the result of a small 2-complex action on a vertex set. The resulted complexes are fully transitive, meaning the automorphism group acts transitively on their faces. 

  Such two-dimensional expanders are rare! Known constructions of such bounded degree two-dimensional expander families are obtained from deep algebraic reasonings (e.g. coset geometries).

  We show that given a small $d$-regular two-dimensional $\epsilon$-expander, there exists an $\epsilon'=\epsilon'(\epsilon)$ and a family of bounded degree two-dimensional simplicial complexes with a number of vertices goes to infinity, such that each complex in the family satisfies the following properties:

\begin{itemize}
\item It is $4d$-regular. 
\item The link of each vertex in the complex is the same regular graph (up to isomorphism).
\item It is $\epsilon'$ expanding.
\item It is transitive.
\end{itemize}

  The family of expanders that we get is explicit if the one-skeleton of the small complex is a complete multipartite graph, and it is random in the case of (almost) general $d$-regular complex. For the randomized construction, we use results on expanding generators in a product of simple Lie groups. This construction is inspired by ideas that occur in the zig-zag product for graphs. It can be seen as a loose two-dimensional analog of the replacement product.
\end{abstract}

\pagebreak
\part{Overview}
\label{over}
\section{General Introduction}
\label{ssub:general-introduction}

The notion of expansion is a key one in graphs, having many applications to a variety of mathematical problems. It generally means how easy is it to get from one vertex set to another one on a random walk on a graph. Among its applications are practical such as error correcting codes \cite{sipser1994expander} which could be used in communication, and theoretical as it was used to prove the PCP theorem \cite{dinur2007pcp}.

Over the last two decades or so, there have been various attempts to generalize this notion to higher dimensions. 

The usual useful notion is that of (pure) simplicial complex: 
\begin{defn}\label{defn:Simplicial_complex}
{\em Simplicial complex } $\mathcal{C}$ is a collection of subsets of set $X$ such that for any  $f \in \mathcal{C}$ , any subset of $f$ is in $\mathcal{C}$. Members of $\mathcal{C}$ are called faces, where the dimension of a face $f$ is $|f|-1$. 

The set of all faces of dimension $k$ is denoted $\mathcal{C}(k)$. Vertices are 0-faces, edges are 1-faces and triangles are 2-faces.
\end{defn}

There has been a growing interest in this field, motivated partially by its usefulness to constructing quantum error correcting codes. It was speculated that it holds the key for solving some theoretical questions that are out of reach for expander graphs otherwise (for exmaple \cite{dinur2017high} ).

The various attempts led to various definitions for expansion in high dimensions.
In graphs, one would find various definitions for expansion that are essentially the same. Expansion in the notion of sets is equivalent to expansion in term of random walks, which is the same as pseudo-randomness. The extension of these properties to the high dimensional case, turned out to yield inherently different notions of expansion, where some of them are even contradicting.

We are interested in a certain expansion property of 2-complexes, that is the (two-dimensional) random walk convergence. 
\begin{defn}[ Random walk ] on a 2-complex $\mathcal{C}$ is defined to be a
    sequence of edges $\mathcal{ E }_{0},\mathcal{ E }_{1},\cdot\cdot\cdot\in \mathcal{C}(1)$ such that\footnote{Original definition by \autocite[]{kaufman2016high}. Based on the definition in \cite{conlon2019hypergraph} }
\begin{enumerate}
	\item $\mathcal{ E }_{0}$ is chosen in some initial probability distribution $p_{0}$
	    on $\mathcal{C}(1)$
	\item for every $i>0$, $\mathcal{E}_{i}$ is chosen uniformly from the neighbors of
	    $\mathcal{E}_{i-1}$. That is the set of $f\in \mathcal{C}(1)$ s.t. $\mathcal{E}_{i-1}\cup f$
	 is in $\mathcal{C}(2)$
\end{enumerate} 
\end{defn}
This is equivalent to a random walk on graph $G_{walk}$: 
    \begin{defn}[Random Walk Graph]
	\label{defn:gwalk}
    The random walk graph of a complex $\mathcal{C}$ ,  $G_{walk}(\mathcal{C})$ is defined by 
        $V(G_{Walk})=\mathcal{C}(1)$  , where 
	$\mathcal{ E }\sim\ \mathcal{ E }'\text{ }\iff\text{ }\exists\ \sigma\in\ \mathcal{C}(2)\text{ s.t. }\mathcal{ E },\mathcal{ E }'\in\ \sigma $
	\begin{center}
	    ($a \sim b$ suggests that $a$ is adjacent to $b$ in the graph)
	\end{center}
    \end{defn}
\begin{defn}[Expander]
\label{sub:expander}
We say $\mathcal{C}$  is a $\epsilon$-expander if $\lambda(G_{ walk })<1-\epsilon$, where $\lambda(G)$ is the maximal non-trivial (normalized\footnote{We assume the matrices are normalized by default }) eigenvalue of the graph. 
\end{defn}
This property is useful in the setting of agreement expanders \cite{dinur2017high}, and has multiple relations to other notions of expansion \cite{kaufman2016high} .

Good expanders  are both sparse, translated to having a low degree, and have good expansion properties. While randomizing edges in graph would typically yield a good expander family \cite{puder2015expansion} , that is not the case in high dimensions. And specifically there are not many ways that are simple and combinatorial in nature.

Until recently, the prime example of bounded-degree high dimensional expander\footnote{We disregard geometric overlapping here} was Ramanujan Complexes (ie \cite{LubotzkySV05}), which yields the best possible expansion properties.

\medskip \medskip We are also interested in symmetric or transitive complexes:

\begin{defn}[Transitive Complex]\label{defn:edgetrantivity}
An automorphism of complex $\mathcal{C}$ is a function 
    \[\phi: V(\mathcal{C}) \longrightarrow V(\mathcal{C})\]
that is bijective and such that $\phi(v) \in \mathcal{C} \iff v \in \mathcal{C}$

A complex $\mathcal{C}$ is called transitive if its automorphism group act transitively on its faces.
for every $x,y \in \mathcal{C}$ where $|x|=|y|$, exists an automorphism $\phi$ s.t. $\phi(x)=y$. 
That means that the automorphism group acts transitively on vertices, edges and so on.

\end{defn}

Such constructions are rare. Some coset geometries are known to be transitive \autocite[]{kaufman2017simplicial}.

There is an open conjecture that suggest that SRW(simple random walk) on a transitive expanders exhibits a cut-off phenomenon\footnote{ See \cite{basu2014characterization} for explanation of this property.  } \cite{lubetzky2016cutoff}. That was proved only for Ramanujan graphs so far (also in \cite{lubetzky2016cutoff}).

\subsection{Combinatorial Constructions of High Dimensional Expanders}
\label{ssub:combinatorial}

The innovation of David Conlon in his construction \cite{conlon2017hypergraph} was that he used combinatorial method to allow one to take random set of generators of Cayley graphs and to provide a 2-complex built upon this that satisfies some HDE expansion properties including the geometric-overlapping property and the 2D random walk.  The drawback of this construction is that it is built upon abliean groups, as they have  a poor expansion properties\footnote{\cite[Proposition 3]{alon1994random} }.

  \cite{chapman2018expander} studied regular graphs whose links are regular.  They achieved some bounds concerning the maximal non-trivial eigenvalue for such graphs. And they constructed 2-expanders from expander graphs whose random walk converge rapidly\footnote{The graphs should be of high girth}.

 \cite{liu2019high} provided a construction called LocalDensifier that takes a small complex $H$ with expansion properties, and a graph $G$ and generates a large complex where all high order random walks have a constant spectral gap. In their construction, the vertex set of the new complex is  $V(G)\times V(H)$.  
They used technique of decomposing the Markov chain into a restriction chain and a projection chain.

In the very recent \cite{alev2020improved}, they obtained a bound on the second eigenvalue in the $k$-dimensional random walk in terms of the maximal second eigenvalue of  the links. Using this method, it is straight forward to obtain a bound on the spectral gap of the LocalDensifier construction.

\section{General Overview}
\label{sub:our_constructions}

We want to take a small general (abstract simplicial-)complex and generate a complex that is {\em transitive regular expander}, and has isomorphic {\em link}. 
We need the following definitions:
    
\begin{defn}
\label{defn:Link}
For complex $\mathcal{C}$:
\begin{itemize}
    \item the {\em link} of a vertex $v$ is the complex $\mathcal{C}_v$ defined by $\{ \sigma \setminus v \mid \sigma \in \mathcal{C}\text{ where }v\in \sigma \}$
     \item The  $k$-skeleton deonted $\mathcal{C}^k$ is  
   $ \{  \sigma \mid |\sigma|=k-1\text{ where }\sigma \in \mathcal{C} \} $
\end{itemize}
\end{defn}
    
\begin{defn}[Upper-regular complex] 

    \cite{LubotzkyLR18} A 2-complex is $d$-{\em regular} if every edge is contained in exactly $k$ triangles     (see \ref{ssub:regular} for more details)
\end{defn}

We want to prove the following theorem:

\begin{thm}
    [Non-formal version of \cref{Lithm} ] 
Given a complex that is $d$-regular, 
with minimal additional requirements on coloring\footnote{\label{rr} $\chi$-strongly-colorable 
(Strong coloring or rainbow coloring means that every triangle has vertices of 3 different colors), $\chi$ is even 
and there is a coloring in which the edges that involves vertices of colors $a,b$  incident to at least 2 distinct vertices of each color },
    exists  $\epsilon'=\epsilon'(\epsilon)$ and a family of  simplicial complexes with a number of vertices goes to infinity, such that each complex in the family has the following properties:
    \begin{itemize}
    \item It is $4d$-regular. 
\item The link of each vertex  is the same regular graph (up to isomorphism).
    \item It is $\epsilon'$ expanding.
    \item It is transitive.
    \end{itemize}

\end{thm}

We will do this in several steps.
First, we introduce a general structure called Schreier complex that represents an action of a complex on a set of vertices. We do so in  \cref{sub:introduction_to_cts}. 

\medskip \medskip Afterwards, we will focus on a private 2-dimensional case of Schreier complex.

	We show a method to analyze expansion of a large complex from its small composing complex assuming they sit together in a structure that we call {\em CTS}(commutative triplets structure). 
   This structure is a special case of Schreier complex in which there is a small 2-complex that acts on a group, and additional conditions are satisfied.

   We do present it in \cref{ssub:cts}, where the exact proofs with all the details are in \cref{parta}.

   We will show that it is possible to generate a CTS from a general\footnote{with the  minimal requirements as in \cref{rr}} 2-complex using a construction we call {\em HDZ}. That in general would only give us an obscure expression for the expansion. In order to make it meaningful, we based our complex on product of simple Lie groups, and rely on expansion properties of random generators in Cayley graphs. We present it in \cref{hdzover}, while the proofs are in \cref{HDZpart}. 

   Finally, we will show that Schreier complexes are transitive when they are based upon a transitive group action, as in the case of HDZ. 

 The proof is short and will be brought here (\cref{symprop}).

\section{Schrier Complex}
\label{sub:introduction_to_cts}

  \medskip \medskip   We remind the reader the definition of Cayley graph:
\begin{defn}\label{defn:cayleygrap}
    Given a group $G$ and a set $S\subset G$,  $Cay(G,S)$ is defined to be the graph with vertex set $G$, and edge set \[\{ ( sg,g) \mid s\in S,g\in G\}\]  
\end{defn}
\begin{defn}[Schreier Graph]
\label{sub:schreier_graph}
    Given a group $G$ that acts on a set $X$ and a set of generator $S \subseteq G,$ the Schreier graph $\operatorname{Shr}(G, X, S)$ is a graph, whose vertices are labeled by elements of $X$ and there is an edge $\left(x_{1}, x_{2}\right)$ iff there exist $s \in S$ such that $x_{2}=s \cdot x_{1}$ (the action of $s$ on $x_{1}$ ).
\end{defn}

Let us focus on the case where $S$ is symmetric and each generator in $S$ is  composed of two commutative steps namely $s= s_1s_2$ where $[s_1,s_2]=0$.
Then $S \subset S_1 S_2$ where $S_1$ and $S_2$ commute .  

\medskip \medskip Let us look at a specific edge \[\mathcal{E}=  \{ s_1s_2g,g \} \]

We wish to describe this edge as a right action of  $G$ upon $P(G)$\footnote{power set of $G$}, or using the following product \begin{align*}
    \Psi : P(G) \times G   &\longrightarrow  P(G)\\
    (A, g)&\longrightarrow A \cdot g = \{ ag \mid a \in A  \} 
\end{align*}
In the product form, our edge has two descriptions\footnote{A fact that would prove beneficial in a generalized construction}:   \[\mathcal{E}= \{s_2,s_1^{-1}\} \cdot  s_1g = \{ s_2^{-1},s_1 \} \cdot s_2g \]

And we can say 

\[E(Cay(G,S))= E( \mathcal{S} ) \cdot g \] where $E(\tilde{G})$ denotes the edge set of graph $\tilde{G}$, and $\mathcal{S}$ a graph with the following edges 
\[ E(\mathcal{S}) =  \{ \{s_1, s_2\} \mid {s_1}^{-1}s_2 \in S , s_i\in S_i  \} \]

\medskip \medskip Now, we can generalize this in two different ways -

We can generalize $\mathcal{S}$ to be a complex instead of graph. And we can generalize the group product to an action of a group on a set.

\begin{defn}[Complex Action on a vertex set]
    \label{defn:psi}
    Let $G$  a group that acts on a set  $X$. 
    Let  $\mathcal{S}$ be a complex s.t. $\mathcal{S} \subset P( G )$. 
    We wish to describe an action of an complex on a vertex set 
    
    by defining the following product:  \begin{align*}
	\Psi &: \mathcal{S} \times X   \longrightarrow  P(X)\\
        \sigma \cdot x  &:= \Psi(\sigma,x)= \{ \sigma \cdot x \mid \sigma \in \mathcal{S}  \} 
    \end{align*}

\end{defn}

\begin{defn}[Schreier complex]
    \label{defn:St} 
    Let $G$  a group that acts on a set  $X$. 
    Let $\mathcal{S}$ be a complex s.t. $\mathcal{S} \subset P( G )$. 
    We define the {\em Schreier complex } of $\mathcal{S}$ on $X$ 
    by 
    \[ 	Sc[\mathcal{S}, X,G] := \mathcal{S} \cdot X= \{ \sigma \cdot x \mid x\in X, \sigma \in \mathcal{S} \}  \]
    using the product $\Psi$.
\end{defn}
\begin{rem*}
    Notice that is not a direct generalization of the Schreier graph, but inspired one.
    If the action is a left product in  a group we will write $Sc[\mathcal{S},G]$.
    If the action is transitive, then the complex is transitive \cref{lem:symmetry}.
\end{rem*}

We think this construction has some potential that has yet to be fulfilled. For example, the Toric code could be interpreted 
as an action of $3$-complex over the set $\mathbb{Z}_m \times \mathbb{Z}_m$ \cite{kitaev2003fault}.

\begin{rem*}
    The CTS construction and all the proofs that are related to expansion, only deal with the case that the action is a group product and that the dimension of the complex is 2. 
\end{rem*}
\section{Commutative Triplets Structure}
\label{ssub:cts}

\begin{defn}[CTS - Shorted  ]
    \label{def:CTS}
    Given a group $G$, and a 2-complex $\mathcal{S}$,
    We define \[\mathcal{C}=Sc[\mathcal{S},G]\] where the involved action is a left group product. 

    Provided that\footnote{ we included the important conditions - see  \cref{def:CTSfor} for a full version}:
    \begin{enumerate}
	\item[A] \label{IT1} $\mathcal{S}$ is $d$-regular
	\item[B] \label{IT2} $\mathcal{S}(1)$ is a collection of commutative generators
	\item[C] \label{IT3} $\mathcal{S}(1)$ is symmetric  
	\item[D] \label{IT4} The action by the complex $\mathcal{S}$ resembles a free action:
	    $\tau \cdot g = \tau' \cdot g'$ only in the trivial case 
	\item[E] \label{IT5} $\mathcal{S}^1$  is connected
	    
    \end{enumerate}

    We say that the Schreier complex $\mathcal{C}$ is also a commutative triplets structure, 
    denoted by\footnote{Using this CTS notation assures additionally that  it satisfies these conditions} $\mathcal{C}=CTS[\mathcal{S},G]$
\end{defn}

\noindent The conditions will be explained shortly. 

Our analysis is based upon describing the walk in terms of equivalent graphs and analyzing them. 
{\em Condition \hyperref[IT1]{A} } would make the required graphs regular, which would allow us to rely upon known theorems of the zig-zag product(\cref{zg}) to analyze them.

As an edge in $\mathcal{C}$ is naturally $\Psi(\tau,g)$ , we describe it by a pair $(center,type)$ as follows:
We think of a function $E: G \times \mathcal{S}(1) \rightarrow V(G_{walk})$ 
  \[E(g,\tau)=\tau \cdot g \] 
  that given a $(center,type)$ translate it into the corresponding edge.
This function is similar to $\Psi$, where the arguments are in reverse order \footnote{Easier to think about this way, and it encompasses standard definition of replacement product. }.

	\begin{defn}[Center \& Type]
    A $2$-edge $w$ is in {\em center} $c$ if \[\exists\tau \in \mathcal{C}(1) \text{ s.t } E(c,\tau)=w\] 
    and $\tau$ would be called the {\em type} of the edge(that is a 2-edge in $\mathcal{S}$).
\end{defn}
Our 2-edge $\mathcal{E}$ from before would be described as of type $\{s_1^{-1},s_2\}$ in center $s_1g$ (equivalently, as type $\{s_2^{-1},s_1 \}$ in center $s_2g$ ).

Each edge is contained in two centers due to the combination of the commutativity requirement({\em Condition \hyperref[IT2]{B}})
and the symmetry requirement ({\em Condition \hyperref[IT3]{C}}). 
The fact that it is in exactly 2 centers is by {\em Condition \hyperref[IT4]{D}}\footnote{It simplifies the analysis but seems non-essential.As more shared centers seems to produce just a better expansion.}.

Sharing centers would be central for the analysis. That would mean that the random walk mix between centers. Between centers the random walk is governed by a graph $G_{dual}$ that is to be introduced. 
And in the vicinity of each center the random walk is governed by the applied local structure that is the small complex $\mathcal{S}$. We need {\em Condition \hyperref[IT5]{E}} for the complex to have expansion properties.

\subsection{Main Theorem}
\label{ssub:main_theorem}

Before we present the main theorem, we will introduce several graphs. 
The first three are enough to present the theorem, while careful examination of the $G_{zig}$ is crucial for the proof. 
\begin{enumerate}
 \item  $L=:G_{walk}(\mathcal{S})$
\item $G_{dual}$ graph

We define $G_{dual}$ as the graph on $G$ with edge set $\{ \{ g,\tau \hat{\cdot} g \} \mid \tau \in \mathcal{S}(1) \} $.

This time we use another kind of product, in which \[\{s_1,s_2\} \hat{\cdot} g := s_1s_2 g \]
This product is well defined because the two generators $ s_1,s_2 $ commute ({\em Condition \hyperref[IT2]{B}}).

 $G_{dual}$ is undirected, as $\mathcal{S}$ is symmetric ({\em Condition \hyperref[IT2]{C}}).
We also have an equivalent definition for $G_{dual}$.

First, we look at the down-up random walk graph $G^{walk} $ (which is also part the Poincaré dual complex). That is a random walk on triangles (through 2-edges)  .

\begin{defn}\label{defn:G^{walk} }
The vertex set of $G^{walk}$ is $\mathcal{S}(2) \times G$ 

Equivalently: \[V(G^{walk}) = \{ \Psi^{-1}(\sigma) \mid \sigma \in \mathcal{S}(2) \}\]

 (i.e. the vertex $\{ a,b,c \},g$ correspond to  $\{ ag,bg,cg \}$ ).

\medskip \medskip An edge exists if the corresponding triangles intersect. 

\end{defn}

We now "forget" the  involves triangle in $\mathcal{S}$, by a projection of $E(G^{walk})$ into its second component($G$). The result is called $G_{dual}$\footnote{The proof of this equivalence is left to the reader. It won't be relied upon. }. Hence, $G_{dual}$ encompasses the edges in $G^{walk}$ that are between two different centers, ignoring their types.

\item $G_{rep}:=G_{dual}\raisebox{.5pt}{\textcircled{\raisebox{-.9pt}{r}}}L$ (detailed definition is \ref{def:We-define-graph})
\label{ssub:_g__rep_}

This is the replacement product  of $G_{rep}$ and $G_{dual}$ \cite{reingold2002entropy}. In general, a replacement product takes a graph $\tilde{G}$ which is $d$-regular, and a graph $\tilde{H}$ on $d$ edges, 
and returns a graph on vertex set $V(\tilde{G})\times V(\tilde{H})$. 

It requires a certain order of the neighbors of $\tilde{g}$. Alternatively,  a function $\phi_v:V(\tilde{H}) \rightarrow \tilde{G} $ that is bijective. 
\[
	 E^{{\color{red}red}}=\{(v,\tau)\sim(v,\tau')\text{ if }\tau \sim \tau'\text{ on }\tilde{H}\}
	 \] 
	\[
		E^{{\color{blue}blue}}=\{(v,\tau)\sim(u,\tau')\ if\ u\sim v\ and\ \text{\ensuremath{\phi_{v}(\tau)=u\ ,\phi_{u}(\tau')=v\}}}
	\]
The edge set of the replacement graph  is $E^{red} \cup E^{blue}$.

In our case,  $\phi_\tau (g) = \tau g$. 
\item $G_{zigzag} = G_{dual}\raisebox{.5pt}{\textcircled{\raisebox{-.9pt}{z}}}L  $ (\cref{sub:zig_zag_graph_}) 
    In general, the zig-zag product of two graphs $G\raisebox{.5pt}{\textcircled{\raisebox{-.9pt}{z}}}H$ is 
    built upon the mentioned replacement product(on the same vertices), by taking all the 3-paths of the form red-blue-red in the corresponding replacement product. This is a well-known and widely used construction aimed at providing expansion while reducing the degree of a large graph \cite{reingold2002entropy}.

    In our case, $G_{zigzag}$ has the same vertices as $G_{rep}$, where its edges are the collection of the red-blue-red edges in $G_{rep}$. 
    
\item $G_{zig}$ 
    The graph defined by the subgraph of $G_{rep}$ that is the collection of all the paths in the form blue-red and red, from any vertex. This is unique to our construction. 
    We also define an operator $T$ that is the adjacency matrix of this graph. 

    \end{enumerate}

\medskip \medskip  We can now present the main theorem of the CTS part:
\begin{restatable*}{thm}{mainthm}\label{main_thm} ( CTS Thoerem) 
    \[\lambda(G_{walk}(\mathcal{C}))\le \sqrt{\frac{1}{2}+\frac{1}{2}\lambda(G_{dual}\raisebox{.5pt}{\textcircled{\raisebox{-.9pt}{z}}}L) }\] 
    where $A \textcircled{\raisebox{-.9pt}{z}} B$ is the zig-zag product between graphs $A$ and $B$.
\end{restatable*}

\label{ssub:expansion_of_zig_zag}

The expansion properties of the zig-zag product are described by the following theorem by Reingoldn, Vadhan and Wigderson (originally \cite[Theorem 4.3]{reingold2002entropy}).
The theorem reads:
	\begin{thm} 
	    \label{zg}
		\begin{math}
		\text {If } G_{1} \text { is an }\left(N_{1}, D_{1}, \lambda_{1}\right)\text{-graph and } G_{2} \text { is a }\left(D_{1}, D_{2}, \lambda_{2}\right)\text {-graph then }\end{math}
		
		$G_{1}\raisebox{.5pt}{\textcircled{\raisebox{-.9pt}{z}}}G_{2}$\begin{math}\text { is a }\left(N_{1} \cdot D_{1}, D_{2}^{2}, f\left(\lambda_{1}, \lambda_{2}\right)\right)-\text{graph, where } \end{math}.
$f$ is a function that satisfies:
\begin{itemize}
    \item $f(a,b)\le a+b$ 
	  \item $f<1$ where $a,b<1$ 
\end{itemize}

	\end{thm}
We call it the "zig-zag" function.
So , we have 

\begin{restatable*}{cor}{corRW}
  \label{corollary:rw}
  $\lambda(G_{walk})\le \sqrt{\frac{1}{2}+\frac{1}{2}f(\alpha,\beta)}$
  where $\alpha=\lambda(G_{dual})$, $\beta=\lambda(G_{walk})$
 and $f$ is the zig-zag function
  
\end{restatable*}
Thus, the expansion properties of $\mathcal{C}$ are somehow the middle ground between the expansion properties of the involved complexes ($\mathcal{S}$ and $G_{dual}$), in a way similar to that of the Zig-zag product. 

We know that some specific important examples are already a CTS, namely:
\begin{itemize}
    \item The construction by Conlon \cite{conlon2018hypergraph}. 
    \item The construction by Chapman, Linal and Peled  \cite{chapman2018expander}  \footnote{Here it refers to a specific case. See \cref{three-product_case}}. 
\end{itemize}
Therefore, we can use this corollary as a "black box". We do so in \cref{partc} for these constructions\footnote{More specifically, \cref{ssub:conlon} for Conlon's and \cref{three-product_case} for the other one }.

\subsection{The course of the proof}
\label{ssub:overview_of_the_paper}
We want to relate the expansion properties of the different presented graphs to the expansion of the complex $\mathcal{C}$. 
We do it (non-formally here) in several steps. First, there are some observation we would like to make.

    We can say the following about $G_{rep}$:
\begin{itemize}
    \item A blue edge connects every two vertices $v,w \in G\times\mathcal{S}(1)$  that are identifiable under $E$ ( $E(v)=E(w)$ ) .
	By \cref{lem:double},  that means that $v=g,\tau$ and $w=\tau g,\tau^{-1}$ for a certain $g\in G$, $\tau \in \mathcal{S}(1)$.
	(It is always true that $E(\tau g,\tau^{-1})=E(g,\tau) $) 
    \item  A red edge  connects $v=g, \tau$ and $w=g,\tau'$, if $E(v),E(w)$ are both in center $g$, 
	and are connected by an edge in $G_{walk}$. That is equivalent to the condition that $\tau$ is adjacent to $\tau'$ in $L$ (by \cref{lem:InGwalk})
\end{itemize}

Therefore, when looking at edge $(v,w)$ of $G_{zig}$. 
If it is a red-edge then $E(v) \sim E(w)$, since they are both in the same center. 
If it is a blue-red-edge then $E(v) \sim E(w)$ too, since the ends of the blue edge are identifiable.
We got a graph homomorphism (\cref{lem:graphhomo}), but there is even a stronger relation.

From here, we will take several steps. 

In step \textbf{1}, we show that   $G_{zig}$ is a lift of $G_{walk}$.

A {\em lift} of graph $\tilde{G}$ is a graph on $V(\tilde{G}) \times A$ for some set $A$, such that the neighborhood of each vertex is kept (More details in \cite{hoory2006expander}). Each neighbor of $v \in V(\tilde{g})$ of $\tilde{G}$ (quotient graph) translates into a neighbor of $v,a$ in the lift (the formal definition is  \ref{liftdef}).

The mapping between the neighborhoods is done by the function $E$ (which is called a covering map).

\begin{restatable*}{lem}{lftleema} \label{liftlemma}
We denote $\Gamma_{ \tilde{G}(v) }$ for the set of neighbors of $v$ in graph $\tilde{G}$.

 For every vertex $v\in V(G_{rep})$, the mapping
 \[ E : \Gamma_{G_{zig}}(v) \longrightarrow \Gamma_{G_{walk}}(E(v)) \]
 is bijective
\end{restatable*}
A main property of a lift is that its spectrum contains  all the eigenvalues of the quotient graph:  

\begin{restatable*}{cor}{corlift}
    \label{corliftA}
    $ \lambda(G_{walk}) \le \lambda( G_{zig} )$ 
\end{restatable*}
In step \textbf{2}, we want to relate $\lambda(G_{zig})$ to $\lambda(G_{zigzag})$. 
Because of the similarity between $G_{zig}$ and $G_{zigzag}$, we can say that two random steps in $G_{zig}$ in probability $1/2$ has the same effect on convergence as a single step in $G_{zigzag}$. We conclude (\cref{lem:In-case-there}) \[ \lambda(G_{zig}) \le \sqrt{\frac{1}{2}+\frac{1}{2}\lambda(G_{zigzag})}\]   
 \textbf{Finally}, we combine them all:

 \[\lambda(G_{walk})\le \lambda(G_{zig})\le \sqrt{\frac{1}{2}+\frac{1}{2}\lambda(G_{dual}\raisebox{.5pt}{\textcircled{\raisebox{-.9pt}{z}}}L)}\]

\section{HDZ Overview}
\label{hdzover}

So far we have discussed CTS. In CTS one gets an expanding complex $\mathcal{C}$ given a complex $\mathcal{S}$, only if $\mathcal{S}$ satisfies demanding conditions, mentioned in \cref{def:CTS}. 
We would like to replicate the required properties in a more general settings, with minimal restrictions on the small complex.

We first assume a general complex $\mathcal{A}$. We do require it to be connected, $d$-regular and $\chi$-strongly colorable
(Strong coloring or rainbow coloring means that every triangle has vertices of 3 different colors).
We choose $\chi$ different groups $G_i$. 
We will generate a complex $\mathcal{C}$  called HDZ that is a special case of CTS. The vertex set of $\mathcal{C}$ would be:
\[\mathcal{G} = G_1\times G_2 \times \ldots \times G_\chi\] 
and that would satisfy all the requirements for CTS.

To do so, we map the vertices of color $i$ in $\mathcal{A}$ bijectively to set of generators $F_i \subset G_i$. 
We have an isomorphic complex $\mathcal{S}$ over the vertex set $\mathcal{F}:= \bigcup_{i\in[\chi]} F_i $, generated by the mapping 
    \begin{align*}
	\phi: \mathcal{A}(0) &\to \mathcal{F}(0) \\
	\phi(V^c_i) &= (F_c)_i
    .\end{align*}

    (More details in \cref{Conv})
    As there are only polychromatic edges in $\mathcal{A}$, there are only edges of different generator sets in $\mathcal{S}$. 
That is every edge in $\mathcal{S}$ satisfies  $\mathcal{E} \in \{F_i,F_j\}$ for some $i \neq j$. 
The fact that it is built upon group tensor product, means that different 2-edges commutate. 
Finally, we define $HDZ^{-}$, the weak version of $HDZ$ to be $\mathcal{C}=Cs[\mathcal{S},\mathcal{G}]$ with left group product in $\mathcal{G}$ as the action. We denote it by  $HDZ^{-}[\mathcal{A},C,\mathcal{G},F_1 , \ldots F_\chi ]$.  For a formal definition see \cref{sub:hdz_procedure}.

We get that {\em Condition \hyperref[IT4]{D}} (the action is freelike action) and {\em Condition \hyperref[IT2]{B}} (the underlying generators of the edges commute) are satisfied automatically. All the other conditions except {\em Condition \hyperref[IT3]{C}} are satisfied too(\cref{isCTS}) . {\em Condition \hyperref[IT3]{C}} translates into a condition we call $Inv$ on $\mathcal{A}$(\cref{propInv}). 

Suppose condition $Inv$ is satisfied for $\mathcal{A}$. Now we get that HDZ is a special case of CTS, 
and we can use the CTS Theorem (\cref{main_thm}) to get the expansion rate based on the expansion of $G_{dual}$  and $G_{walk}(\mathcal{A})$. 
But even if it is not satisfied, using the HPOWER mechanism (next section \& \cref{hpwrsec}) solves this issue. 

The HDZ is inspired by the construction of \cite{chapman2018expander}. And it is in fact a generalization of a special case of it.

\subsection{HPOWER and Property INV}
\label{ssub:hpower_and_property_inv}
(the full details are in \cref{ssub:problem_of_inv_})
For the theorem to be useful, we wish to have complexes that satisfies property $Inv$.
\begin{claim}
We can claim so immediately in several cases:
    \begin{itemize}
        \item Property $Inv$ is satisfied if $A$ is a commutative triplet structure.  
    \item Property $Inv$ is satisfied if the 1-skeleton of $A$ isomorphic to some $Cay(G',F')$ where $G'$ is an abliean group and $F'$ is symmetric. 
\item     Property $Inv$ is satisfied if the complex 1-skeleton is complex $\chi$-partite graph.  
    \end{itemize}
\end{claim}

But in case the complex doesn't satisfy it, we use an auxiliary construction called $HPOWER$. $HPOWER[\mathcal{A}]$ is another complex that maintains the expansion properties of $\mathcal{A}$ while ensuring that condition {\em Condition \hyperref[IT3]{C}} is satisfied.
\begin{defn}
    $HPOWER[\mathcal{A}]$ is a complex defined by  
    \[ E(HPOWER[\mathcal{A}])= \{ 0,1 \} \times E(\mathcal{A}) \]
\end{defn}
We analyze the expansion of $HPOWER[\mathcal{A}]$ and we get that it maintains(\cref{lem:convrate}) the expansion of $\mathcal{A}$ for any complex. The analysis is done by the trace method. 
So we define $HDZ^{+}$ to be the same as $HDZ^{-}$, where $\mathcal{A}$ is replaced by $HPOWER[\mathcal{A}]$ (formal definition in \cref{sub:hdz_procedure}).

\subsection{ Analyzing the expansion of $G_{dual}$ }
\label{ssub:gdualA}

So far, we have assured that all the conditions for the CTS structure are satisfied,
so we can express the expansion rate of HDZ in terms of the expansion of $G_{dual}$ and $G_{walk}(\mathcal{A}$) by using \cref{main_thm}. 
We still don't know the expansion  of $G_{dual}$ in general.  In case the complex contains every possible edge, that means for us that every polychromatic edge is in the graph, that would be easy to handle.

We calculate the eigenvalues explicity and get this theorem: 
\begin{restatable*}{thm}{fullskel}\label{fullskel2}
    Let $\mathcal{A}$ be a $d$-regular complex that is a $d$-regular, strongly $\chi$-colorable 
	(Strong coloring or rainbow coloring means that every triangle has vertices of 3 different colors),
	such that the 1-skeleton of $\mathcal{A}$ is a complete $\chi$-partite graph. 

Let $C: V(\mathcal{A}) \to [\chi]$ be a coloring of $\mathcal{A}$. We denote $V^c$ the vertices of color $c$.
    
Let $Cay(G_1,F_1) \ldots Cay(G_\chi,F_\chi)$ be a collection of Cayley graphs such that 
$|F_c| = |V^c|$ for every $c \in [\chi]$  and such that  \[max(\lambda(Cay(G_i,F_i)))\le \nu \] 

Then,  exists a complex \[\mathcal{C}[\mathcal{A},C,\mathcal{G},F_1 , \ldots F_\chi ]\] which is a $3\mathcal{A}(2)$-regular $2d$-regular transitive with $\mathcal{G}=\prod G_i $ as its vertex set.
And 
\[\lambda(G_{ walk }(\mathcal{C})) \le \sqrt{\frac{1}{2} + \frac{1}{2} f\left(\frac{ N-(\chi-1) + (\chi -1)\nu }{N} ,  \lambda(G_{ walk }(\mathcal{A})) \right)}  \]
where $f$ is the "zig-zag function" \ref{equation:zig-zag} 
and $N={\chi \choose 2}$

This theorem is proved in \cref{applicfull}.
\end{restatable*}
\subsection{Expressing $\lambda(G_{dual})$}
We would like to express the expansion of $\mathcal{G}_{dual}$ by Cayley graph on each coordinate separately. 
However, it is not easy in a more generalize settings. 

We call edges that that look like $\{F_a,F_b \} $, edges of template $ab$. 
We denote the set of edges of template $ab$ by  $S_{ ab }$.

We first want to get an expression for $\lambda(G_{dual})$ based on the graphs $M_{cd}:=Cay(G_c G_d,S_{cd})$. 
To do so, we use an old theorem \autocite[Abstract]{baranyai1979edge} 
that proves that we can decompose the pairs into several unordered partitions of distinct pairs such that any pair appears in one partition exactly once, and the pairs in each partition are disjoint. This allows us to rethink about random walk in $G_{dual}$, in which we first choose a partition and then chose a pair in this partition. 
From this observation, we can calculate the spectral gap of $G_{dual}$ (\cref{allC}). 
\subsection{Lie groups}
\label{sub:lie_groups}

(details in \cref{later})

We want to reduce the condition on $M_{ij}$  to condition on the corresponding 
projection of the edges on $G_i$ and $G_j$.  That is $Cay(G_i,P_k(S_{ij}))$ where $k \in \{ i,j \}$.

$P_k$ is the natural projection( $P_k: G_iG_j \rightarrow G_k$).

We know that this kind of projection maintains expansion properties for product of simple lie groups.
\cite[ Proposition 8.4. ]{breuillard2013expansion} states informally\footnote{Formal relation in \cref{proplie} }  that, 
given $\{s_1,s_2, \ldots s_k\}$ that expand\footnote{ $Cay(G_1,\{s_1, \ldots s_k\})$ is an expander } in $G_1$ and $\{t_1,t_2, \ldots t_k\}$ that expand in $G_2$, assuming they are both simple lie groups with no common factors, then $\{s_1t_1 , \ldots , s_k t_k \}$ expands in $G_1G_2$.

So, if for every $ij$, every $k\in\{i,j\}$, every $Cay(G_i,P_k(S_{ij}))$ is a good expander, we are done. 
Generally, these are different subgraphs of $Cay(G_i,F_i)$. We can't say much generally here. The question is: Is there is a good chance that they will be good expanders?

\cite[Theorem 1.2] {breuillard2013expansion} says(informally)\footnote{It is quoted in \cref{lem:random}} that choosing uniformly randomly 2 generators leads to expanding Cayley graph in probably that is overwhelming\footnote{At least $1-2^{ \Omega(n^c) }$ where $n$ is the number of  vertices in the graph and constant $c$}. 
    
Now, we want to randomize $K$ generators uniformly independently, and we want to get that every pair of them expand.
So, Lovász local lemma comes to our help. We are able to prove the following lemma: 
        \begin{restatable*}{lem}{randlem}\label{lem:random}
    Suppose that G is a finite simple group of Lie type.

        Let $f_1, .. ,f_K \in G$ where $f_i$ are chosen uniformly independently in random, and $K$ is bounded. 
	Then from a certain\footnote{where $|G|\ge N$ } $N$ in probability at least $1- \frac{C}{G^{\delta'}} {K\choose 2}$, $\{ f_i,f_j\}$ is  $\epsilon$-expanding 
	where $C,\delta',N,\epsilon>0$ depend only on $K\text{ and }rk(G)$. 
	  
    \end{restatable*}

We conclude that randomization of elements in product of simple Lie groups could provide a random model for bounded degree complexes. 
We get our main theorem:

\begin{restatable*}[HDZ theorem]{thm}{liethm} \label{Lithm}

Let $\mathcal{A}$ be a complex that is $d$-regular,  $\chi$-strongly-colorable ($\chi$ is even), and the edges that involves vertices of colors $a,b$  
incident to at least 2 distinct vertices of each color \footnote{
   Formally: $|P^c(E^1_{ ab }(\mathcal{A}))|\ge 2$ where 
   \begin{itemize}
       \item $c\in \{a,b\}$
	      \item $E^1_{ab}$ are the edges between colors $a$ and $b$
	      \item $P^i$ is a projection into $V^i$
   \end{itemize}
 }. Suppose it also has a connected 1-skeleton.

Let $C: V(\mathcal{A}) \to [\chi]$ be a coloring of $\mathcal{A}$. $K_c$ is the number of vertices of color $c$ . 

   \medskip \medskip Let $\mathcal{G}  = G_1 \times G_2 \ldots \times G_\chi $ where $G_i$ are product of at most $r$ finite simple (or quasisimple) groups of Lie type of rank at most $r$. Additionally, no simple factor of $G_i$ is isomorphic to a simple factor of $G_j$ for $i\neq j$.

  \medskip \medskip  Let $F_1,F_2 \ldots , F_\chi$ symmetric subsets of the corresponding groups of corresponding sizes $2K_1, \ldots 2K_\chi$ chosen uniformly independently.

  Then in probability at least $1-O(|G_i|^{\delta})$ where $G_i$ is the smallest component of $\mathcal{G}$, and from a certain $N$ ($\forall i\ |G_i|>N$  ), the complex $\mathcal{C}=HDZ^{+}[\mathcal{A},C, \mathcal{G},F_1 , \ldots F_\chi ]$ has the following properties: 

\begin{itemize}
\item Its vertex set is $\mathcal{G}$
\item The degree of each vertex is $24\mathcal{A}(2)$ 
\item It is $4d$-regular. 
\item The link of each vertex  is the same regular graph (up to isomorphism).
\item It is $\epsilon'$ expanding.
\item It is transitive.
\end{itemize}

where $\epsilon',\delta,N$  depend only on ranks of the groups and the choice of $\mathcal{A}$.

\end{restatable*}

The transitivity properties are followed from a standard argument  in \cref{symprop}.
And similarity we get the links are isomorphic and regular in \ref{linksec}.

\section{ Symmetric properties  } 
\label{symprop}

\begin{lem}\label{lem:symmetry}

Suppose that $X=X_1 \times X_2 \times \ldots X_k $ and that 
$F_i \acts X_i$ where this action is transitive (i.e. left group product).
Let $\mathcal{S}$ a $(k-1)$-complex s.t. \[\mathcal{S}({k-1}) \subset \cup_{m\in ( \text{k-tuples in }[\chi] )}  \{F_{m_1},\ldots ,F_{m_k}\}\] Then, the complex $Sc[\mathcal{S},X]$ is transitive. 
\end{lem}
\begin{proof}[Proof Sketch]
    We take $A=\{ a_1, a_2 , \ldots a_\chi \},  B=\{ b_1, \ldots b_\chi \} $  where $A,B \in E(\mathcal{C})$. 
    We want to find $f$ automorphism of $\mathcal{C}$ such that $f(A)=B$. 
    We take $f$ s.t.
    \[ f_i \cdot a_i = b_i\] 
    It is an automorphism because of the way the Schreier complex is constructed $\sigma \cdot X = \sigma \cdot f(X)$ for $\sigma \in \mathcal{S}$. $f_i :X \rightarrow X$ is onto and therefore bijective in every coordinate.
\end{proof}

\begin{cor}\label{cor:issymmetric}
    The HDZ structure is transitive (immediate)
    \end{cor}

\section{Organization of the paper}

\cref{over}  provides an informal overview of the paper, with the key points.  
The reset of the sections should provide the full details of the involved definitions and theorems. 
This is true, except the Schreier complex defined in \ref{sub:introduction_to_cts}

In \ref{parta}, we provide preliminary knowledge and full details for the CTS construction (some of them are also needed for \cref{HDZpart}).

In \ref{HDZpart}, we provide the details for the HDZ construction. 

In \ref{partc}, we provide some additional applications for known constructions. 

\pagebreak
\part{Commutative Triplet Structure}
\label{parta}
\section{Preliminaries}
Here we list different definitions we need, and some notations. 

These definitions are built upon in the current part, and the next one \cref{HDZpart}. 
We repeat here some definitions form the beginning to make the part self-contained.
We do rely on the previously introduced Schreier graph.

\subsection{General Graph}
\begin{defn}
Regarding a $d$-regular expander graph:  
\begin{itemize}
    \item  $G$ is said to be $\epsilon$-expander
\item $\sigma(G)$ is the spectral gap of graph $G$ 
\item $\lambda(G)$ is the normalized nontrivial eigen-value of $G$
\end{itemize}

if the following relation persist:
\[\sigma(G)=d (1-\lambda(G)) = d \epsilon\]  

\end{defn}

\subsection{General Complex}
\label{sub:general_hypergraph}

\begin{defn}[Basic defintions] \label{basiccomplex} Let $\mathcal{C}$ be a complex\footnote{We will mostly deal with 2-dimensional onces so the definitions are not completely general}
    \begin{itemize}
	\item The set of its triangles will be denoted $\mathcal{C}(2)$.
	\item 	The 1-edges of $\mathcal{C}$ are
 \[
     \mathcal{C}(1):=\{\{ u,v \}\ |\ \{ u,v,w \}\in \mathcal{C}(2)\ \text{for some }w\in V(\mathcal{C})\}
	\]
    \item The 1-skeleton of $\mathcal{C}$ is denoted $\mathcal{C}^1$. The vertex set of it is $V$, and the edge set is $\mathcal{C}(1)$.  
    \item The link of a vertex $v$ is the graph $\mathcal{C}_v$ with the edge set \[\{ \mathcal{ E } \setminus v \mid \mathcal{ E } \in E(\mathcal{C})\text{ s.t. }v\in \mathcal{E} \}\]
	\item A complex $\mathcal{C}$ is called 1-connected if $\mathcal{C}^1$ is connected graph
	    \item A complex is of dimension $d$ if the maximally size edge are of size $d+1$
    \end{itemize}
    \begin{rem*}
        We will usually assume the complex is 2-dimensional unless otherwise specified 
    \end{rem*}

\end{defn}
\subsubsection{Regular Complexes}
\label{ssub:regular}

\begin{defn}[Regular]\label{defn:kdregular}

\cite{LubotzkyLR18}
A $d$-dimensional complex is $k$-{\em (upper)regular} if every $( d-1 )$-edge is contained in exactly $k$ $d$-edges. \footnote{A special case of the definition in the index }

A 2-complex is $d$-regular if every edge is contained in exactly $k$ triangles
\end{defn}
There are numerous examples of such complexes. 
\begin{itemize}
    \item Some coest geometries \cite{kaufman2017simplicial} 

	  \item The flag complexes $S(d,q)$  \autocite[section 10.3] {LubotzkyLR18} 
	  \item Conlon's construction \cite{conlon2019hypergraph}
\item Any pseudo-manifold  is an example in which $d=2$. 
	  
\end{itemize}

In the case that the complex has top edges of size $k+1$, and the $k$-edges are exactly $[k] \choose V$, this is a special case of a design with parameters $(|V(\mathcal{C})|,k ,k+1 ,d )$. Designs have been vastly studied, see  \cite{keevash2014existence}.

\subsection{Random Walks}
\label{sub:random_walks}

\begin{defn}[ Random walk ] on a 2-complex $\mathcal{C}$ is defined to be a
	sequence of edges $\mathcal{ E }_{0},\mathcal{ E }_{1},\cdot\cdot\cdot\in \mathcal{C}(1)$ such that\footnote{Original definition by \autocite[]{kaufman2016high} }
\begin{enumerate}
	\item $\mathcal{ E }_{0}$ is chosen in some initial probability distribution $p_{0}$
	    on $\mathcal{C}(1)$
	\item for every $i>0$, $\mathcal{E}_{i}$ is chosen uniformly from the neighbors of
	    $\mathcal{E}_{i-1}$. That is the set of $f\in \mathcal{C}(1)$ s.t. $\mathcal{E}_{i-1}\cup f$
	 is in $\mathcal{C}(2)$
\end{enumerate} 
\end{defn}
This is equivalent to a random walk on graph $G_{walk}$: 
    \begin{defn}[Random walk graph]
    The random walk graph of a complex $\mathcal{C}$,  $G_{walk}(\mathcal{C})$ is defined by 
        $V(G_{Walk})=\mathcal{C}(1)$ , where 
	$\mathcal{ E }\sim\ \mathcal{ E }'\text{ }\iff\text{ }\exists\ \sigma\in\ \mathcal{C}(2)\text{ s.t. }\mathcal{ E },\mathcal{ E }'\in\ \sigma $
	\begin{center}
	    ($a \sim b$ suggests that $a$ is adjacent to $b$ in the graph)
	\end{center}
    \end{defn}

\subsection{Cartesian Product}

\begin{defn}[Cartesiaan Product]
    Given $G,H$ graphs  $G \square H$ is a Cartesian product of graphs where:
    \begin{itemize}
	\item     $V(G \square H) =V(G) \times V(H) $
	\item     $ (u,u') \sim (v,v')$ iff either:
  \begin{itemize}
      \item  $u=v$ and $u' \sim^H v'$ 
	    \item  $u'=v'$ and  $u \sim^G v$ 
  \end{itemize}
    \end{itemize}
\end{defn}
Cartesian product $A \square B$ maintains the lower spectral gap among the graphs: 

\begin{lem}\label{lem:square}
    If $M$,$N$ are $d,d'$ regular graphs
    \[\sigma(M\square N) = \min (\sigma(M) , \sigma(N) )\]
    where $\sigma$ s the spectral gap.
\end{lem}
\begin{proof}
    \[G:=M\square N\] 
    \[D:=deg(G)=d+d'\]
    We have that 
    \[A_G = A_M\otimes I + I \otimes A_N \]
    So, any eigenvector $ v= x\otimes y$ where $x,y$ has eigenvalues $\lambda,\mu$ yields
    \[ A_G v = (d \lambda+d'\mu) v \]
    If we let $  \lambda ,\mu=\lambda(M),\lambda(N) $. 
    We assume wihout loss of generalization
    \[ min \left(\sigma(M), \sigma(N) \right) = \sigma(M)  \]
   Equivalently,  
    \[ \max\{ d \lambda, d'\mu \}=d\lambda \] 
    Then
\begin{align*}
    D\lambda(G) &= d\lambda + d'       = d \lambda + D-d \\ 
    \sigma(G)=D (1 - \lambda(G) ) &= d( 1-\lambda(M))=\sigma(M)  \\
\end{align*}
\end{proof}
\subsubsection{Johnson Graph}
\label{ssub:johnson_graph}

\begin{defn}[Johnson graph]
    \label{jongr}
	$J(S,n)$ is the {\em Johnson graph}

	\begin{center}
		$V(J)={S \choose n}$ and $v\sim v'$ if $|v\cap v'|=n-1$
		\par\end{center}

		For example, in case $n=2$ \\ \[\{a,b\}\sim\{c,d\}\text{ if both sets share one element} \]
\end{defn}

The Johnson graph is a well studied object, and appear naturally when we talk about random walk on complexes. 
It is known that\cite{Filmus_2016} :

\begin{fac}[Spectral Gap of Johnson Graph]
    \label{specjon}
    \[\lambda(J(S,2))=\frac{S-4}{2(S-2)}\]
\end{fac}

\section{CTS definitions}
\subsection{Scherier Complex}
\begin{defn}[Complex Action on a vertex set]
    \label{defn:psi2}
    Let $G$  a group that acts on a set  $X$. 
    Let  $\mathcal{S}$ be a complex s.t. $\mathcal{S} \subset P( G )$. 
    We wish to describe an action of an complex on a vertex set 
    
    by defining the following product:  \begin{align*}
	\Psi &: \mathcal{S} \times X   \longrightarrow  P(X)\\
        \sigma \cdot x  &:= \Psi(\sigma,x)= \{ \sigma \cdot x \mid \sigma \in \mathcal{S}  \} 
    \end{align*}

\end{defn}

\begin{defn}[Schreier complex]
    \label{defn:St2} 
    Let $G$  a group that acts on a set  $X$. 
    Let $\mathcal{S}$ be a complex s.t. $\mathcal{S} \subset P( G )$. 
    We define the {\em Schreier complex } of $\mathcal{S}$ on $X$ 
    by 
    \[ 	Sc[\mathcal{S}, X,G] := \mathcal{S} \cdot X= \{ \sigma \cdot x \mid x\in X, \sigma \in \mathcal{S} \}  \]
    using the product $\Psi$.
\end{defn}
\begin{rem*}
    Notice that is not a direct generalization of the Schreier graph, but inspired one.
    If the action is a left product in  a group we will write $Sc[\mathcal{S},G]$.
    If the action is transitive, then the complex is transitive \cref{lem:symmetry}.
\end{rem*}
\subsection{CTS defintions}

Given $\mathcal{S}$ a complex , and $G$ a group, 
we will describe $\mathcal{C}= CTS[\mathcal{S},G]$ 

\label{sec:cts_defintions}

\begin{defn}[Types]
    $\mathcal{T}=\mathcal{S}(1)$ is the set of types ($\tau=\{\tau_{1},\tau_{2}\}$
	is an element in $\mathcal{T}$). \\
We have the ordered version of it 
\[\mathcal{T}_{o} = \text{\{the 2-ordered-subsets of $\mathcal{S}$\}}\] 

    ($t=(t_{1},t_{2})$ is an element in $\mathcal{T}_{o}$)
\end{defn}

We have a natural inverse in $\mathcal{T}_{o}$. That is
\[
	(a,b)^{-1}:=(b^{-1},a^{-1})
\]
\begin{defn}[E function]
	$E$ gets the edge by center and type 
	\[E :G \times \mathcal{T} \rightarrow \mathcal{C}(1) \]
	\[E(g,\tau)=\{\tau_{1}g,\tau_{2}g\}\]

\end{defn}

\begin{defn}[Center]\label{ndef}

	An edge $w$ is in center $c$ if \[\exists\tau\text{ s.t }E(c,\tau)=w\]
\end{defn}
\begin{defn}[Center function]
   	 $c:E^{1}\rightarrow P(G)$ is a function that returns the set of centers of an edge  
\end{defn}
\begin{defn}[$G_{walk}$]
    $G_{\text{walk}}$ is defined to be the auxiliary graph of the random walk on edges $G_{walk}(\mathcal{C})$.

\end{defn}
\begin{defn}[Graph $L$]
\label{def:graph_l_}
We define $L:=G_{walk}(\mathcal{S})$ 
\end{defn}
\subsection{Formal definition}
\label{sub:conditions} 
 
\begin{defn}[CTS - formal]
    \label{def:CTSfor}
    Given a group $G$, and a 2-complex $\mathcal{S}$,
    We define \[\mathcal{C}=Sc[\mathcal{S},G]\] where the involved action is a left group product. 
    With $\mathcal{T}=\mathcal{S}(1)$, 
    Provided that 
    \begin{enumerate}
	\setcounter{enumi}{-1}
    \item[0]\label{item0} $\{ s,s^{-1}  \} \notin \mathcal{S}\  \forall s\in \mathcal{S}(0) $ 
    \item[A]\label{item1} $\mathcal{S}$ is $d$-regular
    \item[B]\label{item2} $\mathcal{T}$ is a collection of commutative generators
	\[\{a,b\} \in \mathcal{T} \implies ab=ba  \] 
    \item[C]\label{item3} $\mathcal{T}$ is symmetric  \[\{a,b\} \in \mathcal{T} \iff \{a^{-1},b^{-1}\} \in \mathcal{T} \]  
    \item[D]\label{item4} The action by the complex $\mathcal{S}$ resembles a free action.
	    Namely, $\tau \cdot g = \tau' \cdot g'$ only in the trivial case:
	    that is if $\tau'= \tau^{-1}$ and $g'=\tau g $
	    Equivalently: for  $t \neq t' \in \mathcal{T}_o$, s.t.
	 \[t_1 t^{ -1 }_2=t'_1(t'_2)^{-1} \implies\] 
	 \[t'_2=t_1^{-1} \newline \text{ and } t'_1=t_2^{-1} \]

     \item[E]\label{item5} $\mathcal{S}^1$  is connected
	    
    \end{enumerate}

    We say that the Schreier complex $\mathcal{C}$ is also a commutative triplets structure, 
    denoted by\footnote{Using this CTS function assures that it satisfies all the coniditions} $\mathcal{C}=CTS[\mathcal{S},G]$
\end{defn}

\begin{defn}[$G_{ dual }$]

\label{sub:_g__dual_}
    We define $G_{dual}$ to be $Cay(G,\mathcal{ T })$ (aka $Cay(G,\mathcal{S}(1))$ )
    \[
	\{ (g,\tau\hat{\cdot} g)\ |\ \tau \in \mathcal{T} \} 
    \]
    \begin{rem*}
        Here we abuse notation:
    
	For $\tau=\{a,b\}$, $\tau\hat{\cdot} g$ corresponds to $abg=bag$ 
    \end{rem*}
\end{defn}
Notice that since $\mathcal{T}$ is symmetric, this graph is well-defined and undirected.

Here we prove some properties of the structure.
\section{Basic properties}
\label{sub:basic_properties}
We assume all along a commutative triplet structure $\mathcal{S}$ that is $\tilde{d}$-regular ({\em Condition \hyperref[item1]{\textbf{A}}}). 

We look at $L$ (\cref{def:graph_l_}).

$L$ is $2\tilde{d}$ regular because there are $\tilde{d}$ triangles to choose from.
Each introduces 2 distinct edges to choose form\footnote{ Because if $\tau\cup {c}\text{ and }\tau\cup{c'}$ 
both contain $\{a,b\}$, then $\tau=\{a,b\}$.}

As an illustrative example, in case of Conlon (\cref{ssub:conlon}) ,\[L=J(S,2)\] where $J$ is the Johnson graph (\cref{jongr})
Notice that $L$ is a subgraph of the $J(\mathcal{S}(1) ,2) $ graph,
because two elements must have an intersection of size 1,
in order to possibly be adjacent.

We can describe a random walk on the complex
as a random walk on types (that is random walk on $L$), and a random
walk on centers. We intend to construct a graph on the types and centers that would reflect this. 

As in the case of the Conlon's construction, the commutativity requirement
translates into an edge being in two centers. And being in two centers leads
to expansion properties of the complex as the random walk
progresses\footnote{This intuition was largely inherited from Conlon's talk at a conference by the IIAS (Israel institute for advanced studies) in April 2018. He just didn't go as far.}. 

\begin{lem}\label{lem:double}For $\mathcal{C}$ triplet structure
	if $E(g,\tau)=E(g',\tau')$ where $g\neq g'$ then
	\[g'=\tau g\] 
	\[
	\tau' = \tau^{-1}
	\](by abusing notation)
	\begin{rem*}
		We implicitly say that $\tau$ is a symmetric edge.\\
		That is for $\tau=\{t_1,t_2\}$ \[t_1t_2=t_1t_2\] and so \[\tau'=\{t_2^{-1}, t_1^{-1} \}\] 
		and 
		\[g'=t_{2}t_{1}g=t_{1}t_{2}g\]
	\end{rem*}
\end{lem}

\begin{proof}
	Let $t$, $t'$ be the corresponding ordered types.

	That without loss of generalization corresponds to $t_{1}g=t'_{1}g'$ and $t_{2}g=t'_{2}g'$. Of course, if $t=t'$ we get a contradiction.
	\par Then \[t_{1}t_{2}^{-1}=(t'_{1})(t'_{2})^{-1}\]
	and by {\em Condition \hyperref[item4]{\textbf{D}}}, \[t'_2 = (t_1)^{-1}\] \[t'_1=(t_2)^{-1}\]

	we get that $t_{2}t_{1}=t_{1}t_{2}$.
	By
	\[g'=(t'_{2})^{-1}t_{2}g\]
	we get that $g'=\tau g$ and $\tau'=\tau^{-1}$.

\end{proof}
\begin{lem}\label{lem:Every-edge-is}Every edge is in exactly two centers (see definition \ref{ndef} )
 \end{lem}

\begin{proof}
	First $|c(\mathcal{E})|\ge1$ by the definition of $\mathcal{C}$.
	
	But for $t=(a,b)$ by {\em Condition \hyperref[item4]{\textbf{D}}}, we can see that
	\[E(g,t)=E(tg,t^{-1})\]

	Suppose $|c(\mathcal{E})|>2$,
	\[E(g,\tau)=E(g',\tau')\]
	then, $\tau'=\tau^{-1}$.
	If exists another $g'',\tau'$ s.t. $E(g'',\tau')=E(g,\tau)$ then by \lemref{double}
		\[g''=\tau g=g'\] 
		\[ \tau'' = \tau^{-1}=\tau' \]

\end{proof}

\begin{lem}\label{lem:InGwalk} In $G_{walk}$ on $CTS[\mathcal{S},G]$ where it is $\tilde{d}$-regular

\begin{enumerate}
	\item $\mathcal{E} \sim\ \mathcal{E}'$ iff exists $c$ s.t. $\mathcal{E}=E(c,\tau)$ $\mathcal{E}'=E(c,\tau')$ and
	 $\tau\sim^{L}\tau'$
	\item The walk is $4\tilde{d}$ regular

\end{enumerate}
\end{lem}
\begin{proof}\mbox{}\\*
	\begin{enumerate}
		\item Let's check when $\mathcal{E}$ is contained in a
		 certain triangle $\sigma$.

		 The triangle $\sigma$ is $\tilde{s} \cdot g$ where $\tilde{s}\in\mathcal{S}$.
		 and $g$ is distinct for any $\sigma$. 
		 So, any $2$-subset of it is guaranteed to be of the form $\{ s_{a}g,s_{b}g \}$. That means that having a common center is essential for being adjacent in $G_{walk}$. 

		 If $\mathcal{E} \sim \mathcal{E}'$, we must have that $c=c(\mathcal{E})\cap c(\mathcal{E}')$
   exists, so we define $\mathcal{E}'=\mathcal{E}_{o}(c,t')$ and
   $\mathcal{E}=\mathcal{E}_{o}(c,t)$. \\

		 There can't be 2 such centers in the intersection.
		 Suppose they both belong to centers $c,c'$, then $c'=tc$ from \lemref{double}
		 and similarly $c'=t'c$. 

   We define $\tau=\{t_1,t_2\}$ and similarly for $\tau'$. 

		 \begin{center}
			 $\mathcal{ E }=\{\tau_{1}c,\tau_{2}c\}$ $\mathcal{ E }'=\{\tau'_{1}c,\tau'_{2}c\}$ $\sigma=\{sc,s'c,s''c\}$
			
			 \end{center}

		 \par We can see that the condition $\{\tau_{1},\tau_{2}\},\{\tau'_{1},\tau'_{2}\}\subset\{s,s',s''\}$
		 is equivalent to the condition $\mathcal{ E },\mathcal{ E }'\subset\sigma$.
		\item For every center $c$ of $\mathcal{E}$, we have $\mathcal{ E }=E(c,\tau)$ for a certain $\tau$. And for every $\tau'$
		    s.t. $\tau'\sim^{L}\tau$, $\tau'$ induces a distinct edge $\mathcal{ E }'= E(c,\tau')$. 
		    $L$ is $2\tilde{d}$-regular. And we have $d(\mathcal{ E })= |c(\mathcal{ E })|2\tilde{d}$, where $|c(\mathcal{ E })|=2$ by \lemref{double}.

		    Now we assume $\mathcal{E}=E(c,\tau)$, and we want to prove that we got a distinct set of edges.
		    We haven't counted the same edge twice, since that would mean that there exists an edge $\mathcal{E}'=E(c_1,\tau')$ which has two different\footnote{The other case is easy to see} centers $c_1,c_2$ that are also centers\footnote{Each edge has at least one common center with $\mathcal{E}$} of $\mathcal{E}$ . We can assume that $c_1=c$ , and $c_2=\tau' \hat{\cdot} c$ . The other center of $\mathcal{E}$ is $\tau \hat{\cdot} c $ and is equal to $c_2$. Since the generators commutate, $\tau=\tau'$.

	\end{enumerate}
\end{proof}

\subsection{Link of a vertex}
\label{linksec}
(Non-compulsory addition)
Our analysis isn't based upon the link, in contrary to traditional analysis methods. Instead, we will derive the expansion properties from the random walk on the small complex.

We will look at the link of an element in $\mathcal{C}$. We will call this graph $\mathcal{C}_g$ or just $G_{link}(\mathcal{C})$ because they are all isomorphic.

    \begin{lem}\label{lemlink}
	The link of vertex $g$ isomorphic to graph $G_{link}$ 
	\[ V(G_{link} )=\{xy^{-1} \mid \{ x,y \}\in \mathcal{ T }\} \] .
	\[a c^{-1}  \sim b c^{-1}\text{ for every triangle }\{ a,b,c \}\text{ in }\mathcal{S}\] 
    \end{lem}
    \begin{proof}
	Let $h=c^{-1}g$, for every choice of $c\in V(\mathcal{S})$ and $h\in G$.

        The triangles in center $h$ are described by  $\{ ah,bh,ch \}$  for every  $\sigma=\{ a,b,c \} \in \mathcal{S} $. 
	Equivalently:
	\[ \{ a c^{-1} g, b c^{-1} g  , g  \mid g \in G \}\] 	
	
	that would be translated to the edge set of $G_{link}$:
	\[ \{ a c^{-1} , b c^{-1} | \text{ for every triangle }\{ a,b,c \}\text{ in }\mathcal{S}\} \]

	$ac^{-1}=be^{-1} $ only in the trivial case({\em Condition \hyperref[item4]{\textbf{D}}}).

    \end{proof}
	\begin{rem*}
	    Now, the link of a single element $c$ is $\mathcal{S}_c$ and it looks like:
	    
	    \[ \{\{a,b  \} \mid \{a,b,c\}\in \mathcal{S}(2) \}\] 
	    And is isomorphic to 
	    \[ \{\{ac^{-1} ,bc^{-1}  \} \mid \{a,b,c\}\in \mathcal{S}(2) \} \] 
	    So, the link of a single vertex seems like a union of all the links of the vertices in $\mathcal{S}$.
	    It is highly depending upon the structure of $\mathcal{S}$.
	\end{rem*}

\section{Replacement graph properties}
\label{repgraph}
\begin{defn}\label{def:We-define-graph}Given a commutative triplet structure $\mathcal{C}=CTS[G,\mathcal{S}]$,

	We define $G_{rep}$ that stands for a replacement product graph\footnote{The following are standard definitions. Some of them were taken from this excellent lecture about zig-zag product \cite{lect02} }.
	
	\[G_{rep}:=G_{dual}\raisebox{.5pt}{\textcircled{\raisebox{-.9pt}{r}}}L\]
	And more specifically, 	
	in $G_{rep}$, the vertex set is \[V_{rep}=G\times\mathcal{T}\]
	we define for \footnote{Notice that we need that the generators will commute here. The action is defined as in subsection  \ref{def:CTSfor} } $v\in G$, \[\phi_{v}:\mathcal{T}\rightarrow G\]
	\[
		\phi_{v}(\tau)=\tau \hat{\cdot} v = \tau_2 \tau_1 v 
	\]

	We have

	\[v=abu\text{ }\iff\text{ }(u,\{a,b\})\sim(v,\{a^{-1},b^{-1}\}) \]

	\medskip  \medskip and 
	 \[
	 E^{{\color{red}red}}=\{(v,\tau)\sim(v,\tau')\text{ if }\tau \sim \tau'\text{ on }L\}
	 \] 
	\[
		E^{{\color{blue}blue}}=\{(v,\tau)\sim(u,\tau')\ if\ u\sim v\ and\ \text{\ensuremath{\phi_{v}(\tau)=u\ ,\phi_{u}(\tau')=v\}}}
	\]

	\[
		E(G_{rep})=E^{{\color{blue}blue}}\cup E^{{\color{red}red}}
	\]

	$E^{{\color{blue}blue}}$ will be defined by the adjacency matrix $P_{{\color{blue}{\normalcolor }blue}}$ (or simply $P_B$)

	$E^{{\color{red}red}}$ will be defined by the adjacency matrix $P_{{\color{red}red}}$ (or $P_R$)

	$P_{R}|_{g}$ signifies restriction of $P_{R}$ to elements\footnote{That is $\{g,\tau \ \mid\ \tau\in \mathcal{T}\}$} $g,\_$ 
	. This would be of course exactly an instance of the graph $L$.
	So, $P_R$ is $\tilde{d}$-regular.

\end{defn}

\begin{defn}[Zig-Zag graph ]
    We define the zig-zag graph over $V(G_{rep})$ by the operator $P_{R}P_{B}P_{R}$
(see \cite{hoory2006expander}).
\label{sub:zig_zag_graph_}
    
\end{defn}

\section{Operator $T$}
\label{mainpart}

We define $T$ (will act over V($G_{rep}$)) by

\[
	T=\frac{1}{2}P_{R}+\frac{1}{2}P_{R}P_{B}
\]
(the matrices are normalized) 

Thus, $T$ induces a subgraph $G_{zig}$ that is a subgraph of $G_{rep}$.

We define an inverse function to $E:G\times\mathcal{T}\rightarrow E^{1}$ that is
\[\gamma:E^{1}\rightarrow P(G\times\mathcal{T})\]
Notice that $\gamma$ is the labeling function that gives the vertices
in $G_{walk}$ the corresponding name in $G_{rep}$. \newline

The random walk described by $T$ over $G_{rep}$ and the random walk over $G_{walk}$ are very similar,
as demonstrated in the following section:

\subsection{ Relation to $G_{walk}$ } 
\label{formal}

The definition of lift is the following (\autocite[Defintion 6.1]{hoory2006expander}): 
\begin{defn}[Lift] 
    \label{liftdef}
    Let \(G\) and \(H\) be two graphs. We say that a function \(f: V(H) \rightarrow\)
    \(V(G)\) is a covering map if for every \(v \in V(\mathcal{C}), f\) maps the neighbor set \(\Gamma_{H}(v)\) of \(v\) one-to-one and onto \(\Gamma_{G}(f(v))\). If there exists a covering function from $\mathcal{C}$ to $G$,
    we say that $H$ is a lift of $G$ or that $G$ is a quotient of $\mathcal{C}$.
\end{defn}

We are about to prove that $G_{zig}$ is a lift of $G_{walk}$. 
To do so, we show that there is a graph homomorphism between the graphs, and then we show that the neighborhood are transfered bijectively.

\begin{defn}[$G_{zig}$]
\label{sub:_g__semizig_}
    The graph $G_{zig}$ is defined as the induced graph of $T$ on the vertices of $G_{rep}$.
    This is equivalent to the defintion presented 

\end{defn}
\begin{lem}\label{lem:graphhomo}
 There is graph homomorphism between $G_{zig}$ and $G_{walk}$. That is
 the function \[E:V(G_{zig})\rightarrow V(G_{walk})\] (defined earlier) such that if\footnote{There is a directed edge from $g,\tau$ to $g',\tau'$ or the other way around}
 \[g,\tau\rightarrow^{T}g',\tau'\] then \[E(g,\tau)\sim^{G_{walk}}E(g',\tau')\]
 This mapping is 2-1. 
\end{lem}

\begin{proof}
    We denote $e_v$ for vertex $v$ in $G_{rep}$.

 Either\footnote{Because $P_B$ impose a condition on the center} \[e_{g',\tau'} \frac{1}{2} P_R e_{g,\tau}= \frac{1}{4\tilde{d}}\] or \[e_{g',\tau'} \frac{1}{2} P_R P_B e_{g,\tau}= \frac{1}{4\tilde{d}}\] 

In the first case, we have that $g=g'$ and $\tau \sim^L \tau' $. So we are done by lemma \ref{lem:InGwalk} for $c=g$.

In the second case, we have that $\tau g =g'$ and $ \tau^{-1} \sim^L \tau' $. So, by the same lemma, $E(g,\tau) \sim E(\tau g,\tau^{-1})$.
But as mentioned, $E(\tau g,\tau^{-1})=E(g,\tau)$. 

\end{proof}

\lftleema
\begin{proof}
 The mapping is well-defined because it is a graph homomorphism. 
 It is enough to prove that the mapping is onto(because the sets are equal in size, every vertex in both graphs has $4\tilde{d}$ neighbors). 
 Suppose $\mathcal{E} \sim E(v)$ . 
 By lemma \ref{lem:InGwalk} we can assume that $E(v)=E(g,\tau_1)$ and $\mathcal{E}=E(g,\tau_2)$, where $ \tau_1 \sim^L \tau_2$. 
 Therefore, $v$ is either $g,\tau_1$ or $\tau_1 g,\tau_1^{-1}$.
 In any case, it is easy to see that $g,\tau_2 \in \Gamma_{G_{rep}}(v)$

\end{proof} 

The last lemma assured that the $G_{zig}$ is a {\em lift }of the graph $G_{walk}$ as defined by \ref{liftdef}.

It is well known that a lift has all the eigenvalues of the quotient(for example, see \cite{bilu2006lifts}).
We get as a conclusion that $T$ has all the eigenvalues of $G_{walk}$. So, we can deduce:
\corlift 
\section{Global Properties of CS} 

\subsection{Bounding the convergence rate of CTS} 

Now we want to get a bound on the convergence rate of the walk on $G_{walk}$ in terms of the walk on $T$. And then to get a bound on the random walk on $T$. \\

\noindent We define $\pi$ to be the uniform distribution on $V(G_{walk})$.\\
We define $\pi'$ to be the uniform distribution on all vertices of $V(G_{rep})$.\\

\noindent Notice that $\pi'$ on any vertex is half the value of $\pi$. 
Let $V_+$ be the space $x \bot \pi`$ where $x\in \mathbb{R}^{V(G_{rep})}$. We prove here that $T_+$ is well-defined\footnote{$T$ restricted to $V_+$}.

\begin{fac}\label{fact: FactTPI}
The operator T satisfies the following:
\begin{enumerate}
 \item $T\pi'=\pi'$
 \item $V_+(2)\subset V_+$
\end{enumerate}
\end{fac}

\begin{proof}
\mbox{}\\*
 \begin{enumerate}
 \medskip \medskip 
	\item We have \[
	T\pi'=\frac{1}{2}P_{R}\pi'+\frac{1}{2}P_{R}P_{B}\pi'=P_{R} \pi'=\pi'
 \] \\ 
 \par Since $P_{B}\pi'=\pi'$. That is because $\forall x,\tau$ \\
\[
(P_{B}\pi')( e_{ x,\tau } )=\pi' ( e_{ \tau x,\tau^{-1} } )=\pi'( e_{ x,\tau } )
\]

\item Suppose $y\bot \pi'$. Then 

\begin{align*}
 \langle \pi',Ty \rangle &= \langle \pi', P_R y + P_RP_By \rangle \\
	 &= \langle \pi', P_R y \rangle + \langle \pi', P_R P_By \rangle = \langle P_R\pi',y \rangle + \langle P_R \pi', P_By \rangle = \\
	 &= \langle \pi', y \rangle + \langle \pi', P_B y \rangle = \langle P_B \pi', y \rangle = 0 
\end{align*}

 \end{enumerate}\end{proof}

\begin{lem}\label{lem:In-case-there}
    $\lambda(G_{zig}) \le\sqrt{\frac{1}{2}+\frac{1}{2}\lambda(G_{dual}\raisebox{.5pt}{\textcircled{\raisebox{-.9pt}{z}}}L )}$
\end{lem}

\begin{proof}

	\[
		T=\frac{1}{2}P_{R}+\frac{1}{2}P_{R}P_{B}
	\]

	So, we have:

	\[
		T^{2}=\frac{1}{4}[P_{R}^{2}+P_{R}^{2}P_{B}+P_{R}P_{B}P_{R}+P_{R}P_{B}P_{R}P_{B}]
	\]
We will bound $||T^2||_+$ 

	\[
		||P_{R}^{2}+P_{R}^{2}P_{B}||_{+}\le2
	\]

	\[
		||P_{R}P_{B}P_{R}+P_{R}P_{B}P_{R}P_{B}||_{+}\le||P_{R}P_{B}P_{R}||_{+}+||P_{R}P_{B}P_{R}P_{B}||_{+}\le2||P_{R}P_{B}P_{R}||_{+}
	\]
	(Notice that $P_Bx \bot \pi'$ if $x\bot \pi'$)

	\par So we have that

	\[
	    ||T^{2}||_{+}\le\frac{1}{2}+\frac{1}{2} \lambda(G_{dual}\raisebox{.5pt}{\textcircled{\raisebox{-.9pt}{z}}L})  
	\]	
	That assures that we hare a rapid convergence  as the eigenvalues of $T_+^2$ are bounded away from 1.

\end{proof}

We have our main theorem: 

\mainthm

\begin{proof}
 Immediate from corollary \ref{corliftA} and lemma \ref{lem:In-case-there} 
\end{proof}

\corRW
\begin{proof}
We now wish to bound $\lambda(G_{dual}\raisebox{.5pt}{\textcircled{\raisebox{-.9pt}{z}}}L )$ to get a bound on $G_{walk}$. 
We can rely on known theorems about the zig-zag product.
We use here the following theorem by Reingoldn, Vadhan and Wigderson (originally \cite[Theorem 4.3]{reingold2002entropy}).
The theorem reads:
	\begin{thm} 
		\begin{math}
		\text {If } G_{1} \text { is an }\left(N_{1}, D_{1}, \lambda_{1}\right)\text{-graph and } G_{2} \text { is a }\left(D_{1}, D_{2}, \lambda_{2}\right)\text {-graph then }\end{math}
		
		$G_{1}\raisebox{.5pt}{\textcircled{\raisebox{-.9pt}{z}}}G_{2}$\begin{math}\text { is a }\left(N_{1} \cdot D_{1}, D_{2}^{2}, f\left(\lambda_{1}, \lambda_{2}\right)\right)-\text{graph, where } f\left(\lambda_{1}, \lambda_{2}\right) \leq\lambda_{1}+\lambda_{2} \text{ and } f\left(\lambda_{1}, \lambda_{2}\right)<1 \text{ when } \lambda_{1}, \lambda_{2}<1\end{math}.
	\end{thm}
$f$ is the function:
\begin{equation}\tag{ \ding{168} }  \label{equation:zig-zag} 
f\left(\lambda_{1}, \lambda_{2}\right)=\frac{1}{2}\left(1-\lambda_{2}^{2}\right) \lambda_{1}+\frac{1}{2} \sqrt{\left(1-\lambda_{2}^{2}\right)^{2} \lambda_{1}^{2}+4 \lambda_{2}^{2}}
\end{equation}	 
	
We call it the "zig-zag function".

Some of its properties are studied there. It is better (lower) when $\lambda_1$ and $\lambda_2$ are worse. And less than 1 if $\lambda_1$ and $\lambda_2$ are less than 1. This assures the resulted graph is an expander when the original graphs are. 

    \end{proof}    
\begin{restatable}{cor}{comutativegood}\label{corollary:comgood}
Let $G$ be a group.
Let $\mathcal{S}\subset{G \choose 3}$ be a set of triangles. \\
Suppose $\mathcal{C}=CTS[ G,\mathcal{S} ]$ is a $k$-edge regular (where $k$ is bounded by $D$) commutative triplet structure (satisfies conditions \hyperref[itemZ]{\textbf{0}} - \hyperref[itemE]{\textbf{E}} ).
Suppose further that $Cay(G,{\mathcal{S}\choose{2}}) $ is $\epsilon$-expander.
Then, $2D$-random walk on $\mathcal{C}$ converges rapidly with some rate $\alpha(D,\epsilon)<1$.
\end{restatable}

\begin{proof}
    Since $\mathcal{C}=CTS[ G,\mathcal{S} ]$ is of bounded-degree, the number of vertices of $L$ is bounded. By {\em Condition \hyperref[item4]{\textbf{D}}}, $L$ is connected. 
So, each graph $L$ has convergence rate less than 1. 
There are only finitely many possibilities, so there is a number $\beta'(D)<1$ that is the maximal convergence rate for all the graphs $L$. \\

\par We can use corollary \ref{corollary:rw} and get:
\[\lambda(G_{walk})^2 \le \frac{1}{2}+\frac{1}{2}f(1-\epsilon,\beta')<1\]

\end{proof}

\pagebreak
\part{ HDZ } 
\label{HDZpart}

\section{Main Part}
\label{mainpa}

\subsection{HDZ definition}
\label{sub:hdz_procedure}
We describe here the procedure of making the HDZ complex, formally defining the complex $H$
\[H=HDZ[\mathcal{A} ,C, \mathcal{G}, {F_1,F_2, \ldots , F_\chi} ]\] 
given initial complex $\mathcal{A}$ with coloring $C$ and group $\mathcal{G}$ as described above.

\medskip \medskip We have two variants of it $HDZ^{-}$, $HDZ^{+}$ which consists of two, or respectively 3 steps.
\begin{enumerate}

     \item If the original complex doesn't satisfy property $\Inv$ (\cref{propInv}),
 we need to convert it to a complex that does 
 \[\mathcal{A'}=\HPOWER[\mathcal{A}]\]
 while preserving the expansion. See \cref{ssub:problem_of_inv_}.
 ( otherwise $\mathcal{A'}=\mathcal{A}$)

 \item We convert  $\mathcal{A'}$ to complex $\mathcal{S}$ over the group $\mathcal{G}$ that is isomorphic to $\mathcal{A}$. We call it \[\mathcal{S}=CONV[\mathcal{A'},C, {F_1,F_2, \ldots , F_\chi}]\]
     We do so in \cref{Conv} 
\item We simply plug it into the mechanism of commutative triplet structure.
    \[ H= Cts[ \mathcal{G}, \mathcal{S}]  \] 
    We prove it satisfies the required conditions in \cref{CTS}. 
\end{enumerate}

\noindent $HDZ^{+}$ is with the additional step 1.

\noindent $HDZ^{-}$ is without it.

\medskip \medskip To conclude\footnote{By arguments I mean, the rest of the arguments that $HDZ^{-}$ requires } : 
\begin{align*}
        HDZ^{-}[\mathcal{A} ,C, \mathcal{G}, {F_1,F_2, \ldots , F_\chi}] :=& Cts[\mathcal{G},CONV[\mathcal{A},C, {F_1,F_2, \ldots , F_\chi}]] \\
	HDZ^{+}[\mathcal{A} , \text{arguments} ] :=& HDZ^{-}[HPOWER[\mathcal{A}], \text{arguments} ] 
\end{align*}

We take $HDZ$ to be $HDZ^{-} $ if $\mathcal{A}$ satisfies property $\Inv$ and $HDZ^{+}$ otherwise. 
We will specify the needed variant in each case we handle. It has implications on the size and degree of the complex and how similar it is to $\mathcal{A}$.  

The $HDZ^{+}$ variant always works and yields the same convergence rate.

\subsection{The CONV mechanism}
\label{Conv}

Let $\mathcal{A}$ be a complex that is a regular and strongly $\chi$-colorable complex.     (Strong coloring or rainbow coloring means that every triangle has vertices of 3 different colors)

Let $C: V(\mathcal{A}) \to [\chi]$ be a coloring of $\mathcal{A}$. We denote $V^c$ the vertices of color $c$ where $V^c_i$ is $i$-th element of $V^c$. $K_c$ is the number of vertices of color $c$, and we assume it is even. 

   \medskip \medskip Let $\mathcal{G}  = G_1 \times G_2 \ldots \times G_\chi $ 
    \medskip \medskip Suppose we have $F_1,F_2 \ldots , F_\chi$ symmetric subsets of the corresponding groups , where $|F^c|=|V^c|$ for every color $c \in [\chi]$.

We use a useful notation here:

$(c,n)$  is the element of $\mathcal{G}$ that is $F^c_n$.
Given $g=(g_1,g_2,\ldots g_\chi)$, \[(c,n) \cdot (g_1,g_2,\ldots g_\chi) = (g_1 , g_2 ,\ldots ,{F^c_n \cdot g_c}  , \ldots , g_\chi )\]

We use here the notion introduced earlier. 
We define: 

\[\mathcal{F} =  \bigcup_{i \in  I} F_i \]

\begin{defn}[CONV mechanism]
\label{sub:conv_mechanism}

\[\mathcal{S}=CONV[\mathcal{A},C,{F_1,F_2 ,\ldots , F_\chi }] \] is a complex on the vertex set $\mathcal{F}$ 
    generated by the mapping 
    \begin{align*}
	\phi: \mathcal{A}(0) &\to \mathcal{F} \\
	\phi(V^c_i) &= (F_c)_i
    .\end{align*}

\end{defn}

where we assume that $F_c$ is ordered such that it satisfies property $\widetilde{\Inv}$ :
    \begin{defn}\label{tildeinv}
        Property $\widetilde{\Inv}$ is satisfied if for every $i\in [K_c]$:
	\[(F^c_i)^{-1}=\left( F^c_{i+\frac{K_c}{2}} \right)  \]
	And there is no element in $F_c$ of order 2.
	(The indices are taken as modulo $K_c$)

    \end{defn}

    We use this map to generate complex $\mathcal{S}$ on vertices $\mathcal{F}$, where  $\{  a,b,c  \} \rightarrow \{  \phi(a),\phi(b),\phi(c)   \}  $. That is clearly isomorphic to $\mathcal{C}$.

We have this definition: 
    \begin{defn}\label{propInv}
	We say that $\mathcal{A}$ satisfies property $Inv$ if $K_c$ is symmetric for every $c \in [\chi]$, and 
      for every  $c,d\in [\chi]\text{ and } i,j \in [K_c]$ 
    \[\{ V^c_i,V^d_j \} \in \mathcal{A}(1) \iff \{ V^c_{i+\frac{K_c}{2}},V^d_{j+\frac{K_c}{2}} \} \in \mathcal{A}(1)\]
    where the indices are taken modulu $K_c$. 
\end{defn}
        
\begin{defn}
    
    \label{esymmetric}
    An complex $\mathcal{S}$ is edge symmetric if  $\mathcal{S} \subset P(G)$ for some group $G$ and   
\[ \{ a,b \} \in \mathcal{S}(1) \iff \{ a^{-1},b^{-1} \} \in \mathcal{S}(1)  \] 
for the inverse in the group.
\end{defn}

It is easy to see that if complex $\mathcal{A}$ satisfies property $\Inv$, then $\mathcal{S}$ has symmetric edge set.

\subsection{CTS Requirements }
\label{CTS}

\begin{claim}
    \label{isCTS}
    Let $\mathcal{S}$ a $2$-complex s.t. \[\mathcal{S}(2) \subset \cup_{m\in ( \text{3-tuples in }[\chi] )}  \{F_{m_1},\ldots ,F_{m_3}\}\]
    that is edge-symmetric (\cref{esymmetric}) and 1-connected (\cref{basiccomplex}).
    
    We let 
    \[\mathcal{C}=St(G,\mathcal{S})\]
    Then, $\mathcal{C}$ is a standard commutative triplet structure (CTS for short). 
\end{claim}

\begin{proof}
    \medskip \medskip We verify that the required conditions are satisfied.
       We can see that any $\tau \in \mathcal{T}$ is given by $\{ ( n,i ),( m,j ) \}$ for a certain $i\neq j$ where $1\le i,j\le K$ and $n,m$ correspond to $F_m, F_n$. We call $mn$ the template of the edge.
    
   {\em Condition \hyperref[item1]{A}}  is satisfied by the fact that $\mathcal{S}$ is regular.
        {\em Condition \hyperref[item2]{B}} is satisfied because $\mathcal{S}$ is commutative.
        {\em Condition \hyperref[item3]{C}} is satisfied because it  is symmetric.
        There is no element $\{ f,f^{ -1 } \}$ because $i$,$j$ are different ({\em Condition \hyperref[item0]{0}}).   
	{\em Condition \hyperref[item4]{D}} is satisfied because of the following observation. 
       Let's assume 
       \[t_1 t^{ -1 }_2=t'_1(t'_2)^{-1}\]
       Assuming $t$ is of template $ij$ and $t'$ of $i'j'$. 
       Since $t$  effect $g=(g_1,\ldots , g_\chi )$ in the $i,j$ positions, the templates are the same. 
       Then since $F^i$ and $F^j$ commute, the order is not important.
    
       So \[\{ t'_1,t'_2 \} =  \{  t_1 , t_2 \}\] 
    
       And it is easy to see that 
    
       \[t'_2=t_1^{-1} \newline \text{ and } t'_1=t_2^{-1} \]
       as required. 
       {\em Condition \hyperref[item5]{E}} because it is 1-connected.
\end{proof}

\subsection{Basic properties}

\begin{lem}\label{lem:properties}
    Assuming $\mathcal{A}$ is $d$-regular, and   \[\mathcal{C}=HDZ^{-}[\mathcal{A} ,C, \mathcal{G}, {F_1,F_2, \ldots , F_\chi}]\] is a valid CTS. Then $\mathcal{C}$ has the following properties: 
\begin{enumerate}

    \item $|\mathcal{G}|$ vertices
    \item  The degree of each vertex is $3|\mathcal{A}(2)|$
    \item  $2d$-regular
\end{enumerate}
\end{lem}

\begin{proof}
    $\mathcal{S}$ is isomorphic to $\mathcal{A}$. 
    
    $V(\mathcal{C})=\mathcal{G}$ by definition.

    Let's look at vertex $x\in \mathcal{G}$. Every selection of $\tilde{ s } \in \mathcal{S}(2)$ yields 3 different options for a center of edge. 
    Namely, \[c= {\tilde{s}_i}^{-1} x\text{ for }i\in \{ 1,2,3 \}\] 
    This induces $\tilde{s} c$ as the triangle that contains it (by abuse of notion). 

    These are all the options.  So, there are $3|\mathcal{A}(2)|$ options in total.

    The regularity of the random walk on $\mathcal{C}$ is  4 times the regularity of $\mathcal{S}$. As the random walk includes two new edges for every triangle, the regularity of $\mathcal{C}$ is $2d$. This is also clear, since an edge is in two centers, and each induces $d$ distinct triplets that contains the edge. 
\end{proof}

\subsection{The convergence rate}
\label{defmu}
Now, $G_{dual}$ is defined to be the graph (\cref{sub:_g__dual_})
\[G_{dual}=Cay(\mathcal{G},\mathcal{T})\]
Where \[\mathcal{T}=\mathcal{S}(1) \]
This is an abuse of notation.

Another way to look at it, the mapping $\phi$ induces a function $\mu$ that relate the vertices of $\mathcal{A}(1)$ to $\mathcal{F}^2$ defined by

\[\mu( \{a,b\})= \phi_{c(a)} (a)  \phi_{c(b)}(b) \]

And we have \[G_{dual}=Cay(G,\mu(\mathcal{A}(1)))\]

We have proved that this is a commutative triplet structure. So we can finally use the main theorem , and get:

\[\lambda(G_{walk}(\mathcal{C}))\le \sqrt{\frac{1}{2}+\frac{1}{2}\lambda(G_{dual}\raisebox{.5pt}{\textcircled{\raisebox{-.9pt}{z}}}G_{walk}\left( \mathcal{A} \right) ) }\] 

We will combine all that we have concluded so far: 
\begin{prop}\label{propwithinv}
	Let $\mathcal{A}$ be a complex that is a regular and strongly $\chi$-colorable complex. 
	(Strong coloring or rainbow coloring means that every triangle has vertices of 3 different colors)

Let $C: V(\mathcal{A}) \to [\chi]$ be a coloring of $\mathcal{A}$. We denote $V^c$ the vertices of color $c$ where $V^c_i$ is $i$-th element of $V^c$. 
$K_c$ is the number of vertices of color $c$ (and it is even).

   \medskip \medskip Let $\mathcal{G}  = G_1 \times G_2 \ldots \times G_\chi $ where $G_i$ are groups. 
    \medskip \medskip Suppose we have $F_1,F_2 \ldots, F_\chi$  subsets of the corresponding groups.

\medskip \medskip     We require that:
\begin{itemize}
    \item $|F_c|=K_c$ 
    
    \item $\mathcal{A}$ satisfies property $\Inv$(definition \ref{propInv}).
\end{itemize} 

The complex that is defined by  \[\mathcal{C}=HDZ^{-}[\mathcal{A} ,C, \mathcal{G}, {F_1,F_2, \ldots, F_\chi} ]\] satisfies:
\[\lambda(G_{ walk }(\mathcal{C})) \le \sqrt{\frac{1}{2} + \frac{1}{2} f\left(\lambda(G_{dual}) ,  \lambda(G_{ walk }(\mathcal{A})) \right)}  \]
where (for $\mu$ defined earlier)
\[
    G_{dual}=Cay(\mathcal{G},\mu(\mathcal{A}(1)))
\]

and $f$ is the "zig-zag function" \ref{equation:zig-zag}

\end{prop}

\subsubsection{Two problems}
\label{ssub:two_problems}

To successfully use what we have so far, we need two conditions. The first is that $Inv$ property should be satisfied.
We will handle this in the section \ref{ssub:problem_of_inv_}. 

The second hurdle is that we can't always assure that $G_{dual}$ is a good enough expander. However, in certain cases, we can.

We reduce the expression for the expansion of $G_{dual}$ to something that is occasionally more manageable\footnote{At least in one case}  in section \ref{ssub:analyzing_the_expansion_of_g__cay}.
And we get a concrete result in \cref{later}.

\section{Full 1-skeleton case}
\label{applicfull}

We calculate $\lambda(G_{dual})$ in the case the complex has full 1-skeleton.
Since the complex is strongly $\chi$-colorable, full means that all the possible $2$-edges are present, namely, that the 1-skeleton of it is a $\chi$-multipartite graph. 

In this case, We have no problem with $Inv$ or with calculating $G_{dual}$ .
We know property $Inv$ is satisfied as it contains the required edges.

\begin{lem}\label{lem:full1skellam}

    Let $\mathcal{A}$ be a $d$-regular, strongly $\chi$-colorable 2-complex. 
	(Strong coloring or rainbow coloring means that every triangle has vertices of 3 different colors)

Let $C: V(\mathcal{A}) \to [\chi]$ be a coloring of $\mathcal{A}$. We denote $V^c$ the vertices of color $c$.
    
Let $Cay(G_1,F_1) \ldots Cay(G_\chi,F_\chi)$ be a collection of cayley graphs such that 
$|F_c| = |V^c|$ for every $c \in [\chi]$. With $N={\chi \choose 2}$ , we have
\[ \lambda(G_{dual}) = \frac{ N-(\chi-1) + (\chi -1)\nu }{N}\]

\end{lem}

\begin{proof}

We know that \[\mu(A(1))=\mathcal{F}^2\]
\[\mathcal{F}^2 := \cup_{i\neq j\in [\chi]} F_iF_j\]

So 
\[G_{dual}= Cay(\mathcal{G}, \mathcal{F}^2)\]

Therefore if we set \[A_i = Cay(G_i,S_i)\]
Then we can define:
\[A_{ij}:=I \times \ldots  \times \overarrow[i]{A_i} \times I \times \ldots \times \overarrow[j]{A_j} \times \ldots \times I\]

that is the adjacency matrix of $Cay(\mathcal{G},S_{ ij })$ in our case, where \[S_{ ij } = S_i \times S_j\]
And $G_{dual}$ is just a union of all possible $Cay(\mathcal{G},S_{ ij })$ . 

\medskip \medskip We let $N= { \chi \choose 2 }$ 
\[G_{dual}= \frac{1}{N} \sum_{\{ i,j \} \in { \chi \choose 2 }} A_{ij}\] 
We let $y=x_1\otimes x_2 \ldots \otimes x_\chi $ be a collection of eigenvectors, where $x_i$ is an eigenvector of $A_i$ with $\mu_i$ as its eigenvalue. 

 \[G_{dual}y =\frac{1}{N} \sum_{\{ i,j \} \in { \chi \choose 2 }} A_{ij} y = \frac{1}{N} \sum_{\{ i,j \} \in { \chi \choose 2 }} \mu_i \mu_j y \] 
That is why the eigenvalue is \[\frac{1}{N}\left( \sum_{\{ i,j \} \in { \chi \choose 2 }} \mu_i \mu_j \right)\] 
Now, to maximize it we choose all the eigenvalue to be 1, except $\mu_1$ which would be $\max_i\left(\lambda_2( A_i )\right) $ .
The number of times $\mu_1$ appears in the sum is $\chi-1$.

So, we have
\[\frac{ N-(\chi-1) + (\chi -1)\nu }{N}\]

as the maximal eigenvalue.

\end{proof}\medskip \medskip We now have 
\fullskel

\begin{proof}
    We use \cref{propwithinv} with $\mathcal{A}$.

    We know property $Inv$ is satisfied. And it is already $d$-regular.
    The properties are from \cref{lem:properties} .

    We have $\lambda(G_{dual})$ from \cref{lem:full1skellam}
\end{proof}

An interesting application of this is the case the complex is complete(complete 3-partite graph). This case is called the 3-product case.

\section{Property $\Inv$ and HPOWER }
\label{ssub:problem_of_inv_}

For the theorem to be useful, we wish to have complexs that satisfies property $Inv$.
We can claim so immediately in several cases. 
\begin{claim}
    Property $Inv$ is satisfied if $A$ is a commutative triplet structure.  
\end{claim}
\begin{proof}
    Obvious, since it is one of the requirements ({\em Condition \hyperref[item3]{C}}).
\end{proof}

What is interesting about this claim is that, we can composite several HDZ constructions together, or start with a known CTS, such as Conlon's construction, and continue with HDZ, yielding larger and larger constructions. 

\begin{claim}
    Property $Inv$ is satisfied if the 1-skeleton of $A$ isomorphic to some $Cay(G',F')$ where $G'$ is an abliean group and $F'$ is symmetric.
\end{claim}
\begin{proof}
    We prove that if $(g,g')\in A(1)$ then $(g^{-1},g'^{-1}) \in A(1)$.
    In this case, \[g'=fg\] for $f\in F'$. And surely  \[(fg)^{-1}=f^{-1}g^{-1}=g'^{ -1 }\]
    Since $f$ symmetric, $f^{-1} \in F'$
\end{proof}

\subsection{ HPOWER } 
\label{hpwrsec}
\label{General complex}
Even if we can't prove our original complex $\mathcal{A}$ satisfies $\Inv$, we can just generate a new complex from $\mathcal{A}$ that does. What we do is we define a power of complex that is similar to normal product by a full complex, and we prove it preserves the expansion.
We only need to do product by complete complex of 2 vertices, but it could be easily extended to complete complex of general $k$-vertices. And possibly could be a useful tool in other situations.
\begin{claim}\label{claim:HPOWER}
    For any $d$-regular complex $\mathcal{A}$ s.t. $G_{walk}(\mathcal{A})$ is $\epsilon$-expander, we define  \[\mathcal{A'}=HPOWER[\mathcal{A}]\] that has the following properties:
    \begin{itemize}
	\item  $|V(\mathcal{A}')|=2|V(\mathcal{A})|$.
	\item  $|\mathcal{A}'(1)|=4|\mathcal{A}(1)|$.
	\item  $|\mathcal{A}'(2)|=8|\mathcal{A}(2)|$.
	\item  $(\mathcal{A}')$ is 2$d$  edge regular
	\item $G_{walk}(\mathcal{A}')$  is $\epsilon$-expander
    \end{itemize}
\end{claim}
\begin{defn}
    $\mathcal{A}'$  is over $(0,1)\times V(\mathcal{A})$: 

    For every $e \in E(\mathcal{A})$ \[\{ 0,1 \} \times e \in E(\mathcal{A'})\]

\end{defn}

\medskip \medskip

We define
\[ V:=V(\mathcal{A})\text{  }V':=V(\mathcal{A}')\]
We think of $G_{Walk}(\mathcal{A'})$ as a graph on vertex set \[U':=V(G_{Walk}')= { V'\choose{2} }\] The walk is by the matrix $A'_{walk}$.
And similarly for $G_{Walk} (\mathcal{A})$.

\begin{defn}
    We define a projection $p:E(\mathcal{A}) \rightarrow E(\mathcal{A}')$ \[x\times e \rightarrow e\] where $x\in\{ 0,1 \}, e\in E(\mathcal{A})$.

    Similarly, We define a projection  $q:E(\mathcal{A}) \rightarrow \{ 0,1 \}$ \[x\times e \rightarrow x\]
\end{defn}

\subsection{Properties of HPOWER}
\label{sub:PropertiesOfMakeInv}

The first 3 claims are obvious.

Let $J$ be the Johnson graph  $ J( V',2)$

\begin{claim}
    For $v,w \in U'$ 
    \[v \sim^{A'_{walk}} w  \iff v \sim^J w \text{ and } pv \sim^ { A_{walk} }  pw \] 
\end{claim}

Let's explain the condition.

Assuming \[\{(a,x),(b,y)\} \sim \{(a',x'),(b',y')\}\] 
If the walk is valid, one of the vertices of $U'$ is common to both of them.
This is equal to the condition that should be adjacent in $J$.

The second condition is that under $p$, they are adjacent in $A$.
It is true if \[\{x,y\} \sim^A  \{x',y'\}\] We don't demand any conditions on $a,a',b',b$ as the edges are defined as a tensor product on edges. 

\begin{lem}\label{lem:cycles}
    With complex $\mathcal{A}$ and $\mathcal{A'}=HPOWER[\mathcal{A}]$ 

    \[ tr(A'^{2m})=tr( A^{2m})2^{2m+1}\] 
\end{lem}

\begin{proof}
     Now, we inspect the circles of length $2m$ of $A'_{walk}$. 
    We can easily see that given \[u_1, \ldots u_{2m+1}=u_1 \in U'\text{ a circle in }A'_{walk}\] then \[p u_1, \ldots p u_{ 2m+1 }= pu_1 \in U\text{ a circle in }A_{ walk }\]
    
    So, this is a necessary condition. We will see that it is also sufficient.
    
    Let $u_1 \ldots u_{2m+1}=u_1 \in U'\text{ a circle in }A'_{walk}$.

    We assume \[u_j=\{ (a,x),(b,y) \}\] 
    As long as one element in $u_j$ is kept, and the condition on $p u_{j+1}$ is satisfied, the step is legal. 
    
    Assuming we keep $(a,x)$ , the next one could be $\{ (a,x),(0/1,z) \}$, provided that $\{ x,y,z \}$ is a valid triangle in $A_{walk}$.
    So, we can decide on $q( u_{j+1} \setminus u_j ) $ where $j\in 2\ldots 2m$ . That means we have freedom of $2$ choices per step. 
    We can also see that if the random walk on $\mathcal{A}$ is $d$-regular, the random walk on $A'$ is $2d$-regular.
    
    All in all for every $u$ \[{A'}^{2m}_{uu}= 2^{2m-1} A^{2m}_{pu,pu} \]
    Now
    \[ \sum_{u\in U'} {A'}^{2m}_{uu}= \sum_{u\in U'} 2^{2m-1} A^{2m}_{pu,pu}  \]
    We notice that in the sum every $A^{2m}_{ vv }$ for every $v\in U$ is obtained $4$ times.
    \[ = \sum_{v \in U} 2^{2m+1} A^{2m}_{vv} \]  
    
    Therefore, we have that \[tr(A'^{2m})=tr( A^{2m})2^{2m+1}\]
\end{proof}

\begin{lem}\label{lem:convrate}
    With complex $\mathcal{A}$ and $\mathcal{A'}=HPOWER[\mathcal{A}]$ 
   \[ \lambda(G_{walk}(\mathcal{A}))=\lambda(G_{walk}(\mathcal{A})\]
    \end{lem}
    \begin{proof}
        Lets take $P'=\frac{A'}{2d}$, the normalized version of $A'$.
	\[
	    tr({ P' }^{2m})= \sqrt{ \sum { \lambda' }_i^{2m} } = \sqrt{ {\lambda}'^{2m}(1+o(1)) } = {\lambda'}^m+o({ \lambda' }^m)
	\]
	\medskip \medskip This is also true for $P$.
	
	And \[tr(P'^{2m})=\frac{tr({A'}^{2m})}{{ ( 2d ) }^{ 2m  } } = \frac{tr(P^{2m}) 2^{2m+1} }{2^{2m}} = tr(P^{2m}) 2   \] 
	So,
	
	\[ \lambda' = \lim_{m \rightarrow \infty} \sqrt[m]{tr(P'^{2m})} = \lim_{m \rightarrow \infty} \sqrt[m]{2} \sqrt[m]{tr(P^{2m})} = \lambda  \]
    \end{proof}

\begin{claim}
    There is a coloring $C$, order on $V$ such that complex $\mathcal{A}$ satisfies property $\Inv$ (\cref{propInv}).

\end{claim}
\begin{proof}
    We can see that if $a\in V$ then there is no triangle that contains $\{ (0,a),(1,a) \} $, as this would suggest that $\{ a,a \}$ is an edge in $ \mathcal{A}$.
    Therefore, given a vertex $V^c_i$, we define  \[V'^c_i=(0,a)\] 
    \[V'^c_{i+K_c} =(1,a)\]  This is a valid coloring. The reason is that if $\{ (x,a),(y,b),(z,c) \}$ is a triangle iff $\{ a,b,c \}$ is a triple. 
    It is straight forward to see that property $\Inv$ is satisfied.
\end{proof}

\section{Reducing $G_{dual}$}
\label{ssub:analyzing_the_expansion_of_g__cay}

In this section we analyze the expansion properties of $G_{dual}$.
\subsubsection{Definitions}
\label{ssub:definitions}

We have several definitions here.

We call $S_{ cd }$ the generators obtained by the edges of index (originally colors) $c,d \in [\chi]$. 
\[S_{cd}:=\{  (c,i) \cdot (d,j) \mid \  \{(c,i),(d,j)\} \in \mathcal{S}(1) \} \]  
And 
\[  S^{ijk}:= \bigcup_{\{c,d\} \in \{i,j,k\}} S_{cd} \]

\[M_{cd} := Cay(G_{ c}  G_{d },S_{ cd })\]
\[M^k_{cd} := Cay(G_{ c}  G_{d },P_k(S_{ cd }))\]
where $P_k$ is the natural projection $P_k : G_c G_d \rightarrow G_k$ 
for $k \in \{c,d\}$.

\begin{lem}
    \label{allC}
Let $N= \chi-1$.
   Exists $U_1 \ldots U_N$ sets of 2-elements in ${[\chi] \choose{2}}$ such that \[\bigcupdot U_i= {[\chi] \choose 2}\] and each 2-element appears once in one of the $U_i$s s.t.

    \[ \sigma(G_{dual}) \ge 
    \sum_{l=1}^N    \min_{ij \in U_l} \left( \sigma(M_{ij})  \right)   \]
\end{lem}
\begin{proof}
    We quote from \autocite[Abstract]{baranyai1979edge}

   \begin{displayquote} If $h | n$ then the $h$ -element subsets of an $n$ -element set can be partitioned into $\left(\begin{array}{l}n-1 \\ h-1\end{array}\right)$ classes so that every class contains $n / h$ disjoint $h$ -element sets and every $h$ -element set appears in exactly one class. \footnote{I would like to thank Liran Katzir for pointing me to the paper which has precisely the wanted claim} \end{displayquote} 

 \medskip \medskip   Lets apply this for $h=2$ and $n=\chi$.We set $N= {\chi-1}$.
   Using the theorem, we get $U_1 \ldots U_N$ sets such that \[ \bigcupdot U_i= {[\chi] \choose 2}\]
   And such that for $U_i = \{ u_1 , u_2 \ldots u_{\frac{\chi}{2}}  \}$ \[   u_1 \cupdot u_2 \ldots \cupdot u_{\frac{\chi}{2}}  = [\chi] \] 
   For each set $U_l$, we relate a graph
\[ 
    G^{l}:= \square_{ \{ ij\} \in U_l   }M_{ij}
\]
Notice that $G^{l}$ is on the same vertices as $G_{dual}$. 

  \medskip \medskip  On a r.w. on $G_{dual}$,  we chose uniformly an edge $u$ from $S_{cd}$ in probability that is $\propto |S_{cd}|$ and move to the incident vertex.

\begin{claim}
    The last method of selection is further equivalent to the following:

Choosing $U_k$ in probability $\propto \sum_{\{lm\} \in U_k} |S^{lm} |$, then choosing an edge uniformly from $G_{l}$. 

\end{claim}

\medskip \medskip    All in all we concluded that
    \[A_{cay}= \sum_{k=1}^N a_k A^{k}  \]
    for \[a_k= \frac{\sum_{ \{lm\} \in U_k} |S^{lm} |}{d(G_{dual})} \]
    Notice that all $A^{k}$ are distinct.
    Generally (\cref{lem:square}): 
    \[\sigma(A\square B) = \min (\sigma(A) , \sigma(B) )\]
    Where $\sigma$ is the spectral gap of the graph $\sigma(A)$ . And specifically,

\[\sigma(G^{l})= \min_{\{ ij\} \in U_l  }   \sigma\left( C^{ij} \right)\]

    We want to plug it in.
    \begin{align*}
	\lambda(G_{dual}) \le& \sum_{l=1}^N a_l \lambda(G^{l}) \\ 
	1-\lambda(G_{dual}) \ge& \sum_{l=1}^N a_l (1-\lambda(G^{l})) \\
       =& \sum_{l=1}^N \frac{d(G^{l})}{d(G_{dual})} (1-\lambda(G^{l})) \\
    \end{align*}
    So 
    \[	\sigma(G_{dual}) \ge \sum_{l=1}^N  \sigma(G^{l})\]
    
    And finally,
    \[ \sigma(G_{dual}) \ge 
    \sum_{l=1}^N   \min_{ij \in U_l} \left( \sigma(M_{ij}) \right)   \]

\end{proof}

We state the lemma that we got:

\begin{lem}\label{genthmA}

   \medskip \medskip Let $\mathcal{G}  = G_1 \times G_2 \ldots \times G_\chi $ where $G_i$ are groups. 
   \medskip \medskip Suppose we have $F_1,F_2 \ldots , F_\chi$ symmetric subsets of the corresponding groups

   Let $\mathcal{S}$ a 2-complex s.t. \[\mathcal{S}(2) \subset \cup_{m\in ( {[\chi] \choose{3}} )}  \{F_{m_1},F_{m_2} ,F_{m_3}\}\] 

For simplicity, we assume $2 \mid \chi$ and $\chi\ge 3$ .
    
\medskip \medskip     Suppose that $\exists \nu<1$ s.t. $\forall ij \in  { [\chi] \choose{2} }$

\begin{equation}\tag{\ding{72}} \label{eq:einstein}
\min \left( \sigma(Cay(G_iG_j,S_{ ij }))  \right) \ge \frac{(1-\nu)|\mathcal{S}(1)|}{\chi-1} \\
 \end{equation}	

 Then, the complex that is defined by  $\mathcal{C}=Sc[\mathcal{S}, \mathcal{G}] $ satisfies:

 \[\lambda(\mathcal{C}) \le \sqrt{\frac{1}{2} + \frac{1}{2} f\left( \nu ,  \lambda(G_{ walk }(\mathcal{A})) \right)}  \]

\end{lem}

\begin{proof}
If condition \ref{eq:einstein} is true, then all the graphs $G^ {k}$ has spectral gap of at least \[\frac{(1-\nu)|\mathcal{S}(1)|}{{{\chi-1}\choose{2}}}\] 

Then, we have by \cref{genthmA}: 
\[ \sigma(G_{dual}) \ge \sum_{ l=1}^N \sigma(G^{l}) \ge (1-\nu) |\mathcal{S}(1)|  \]
\[ \lambda(G_{dual}) \le \nu \]

\end{proof}

\section{Random model on Lie groups }

\label{later}
Unfortunately, we have no out of the box way to verify that indeed condition \ref{eq:einstein} is satisfied. But we can verify it in case of Lie groups when we randomize elements.

In this section, we use the notation as appears in \cite{tao2015expansion}:

\begin{defn}[Expanding set] \label{defn:eps-expanding}
    $\{ a,b \}$ is $\epsilon$-expanding for group $G$ if $Cay(G,\{a,b,a^{-1},b^{-1}\})$ is $\epsilon$-expanding. 
\end{defn}

First we proof this lemma: 
\randlem

   \begin{proof}
       We randomize $f_1 \ldots f_K$ and we want that every pair of $\{f_i,f_j\}$ would be $\epsilon$-expanding.
       For this, we use the following theorem:
\begin{displayquote}
    \begin{thm}\label{rpairs}
        ( \cite[Theorem 1.2] {breuillard2013expansion} Random pairs of elements are expanding). Suppose that G is a finite simple group of Lie type and that \(a, b \in G\) are selected uniformly at random. Then with probability at least \(1-C|G|^{-\delta},\{a, b\}\) is $\epsilon$-expanding for some \( C, \epsilon, \delta>0 \) depending only on the rank of \(G .\)
    
    \end{thm}
\end{displayquote}

And we combine it with the asymmetric case of Lovász local lemma\footnote{We remind it to the reader later} \autocite[lemma 5.1.1 on pg.64]{alon2004probabilistic} . 

       \medskip \medskip Let $E_{ij}$ the event in which $Cay(G,\{ f_i,f_j \})$ doesn't generate $G$ for this $\epsilon(rk(G))$.

              Then \[\Pr(E_{ij})\le C|G|^{-\delta}\]
	      We define a dependency graph of $E_{ij}$ (we identify it with the vertex $\{i,j\}$). This is essentially Johnson graph(\cref{jongr}) $J(K,2)$, where  \[\{ i,j \} \sim \{ i,k \}\text{ if }k\neq j\] \[\{ i,j \} \sim \{ j,k \}\text{ if }k \neq i\]
	      It has this form because $E_{ij}$ is mutually independent of $E_{kl}$ where $\{k,l\} \cap \{i,j\} = \emptyset$.
That means that every vertex $E_{ij}$ has $2(K-1)$ neighbors.

\medskip \medskip Now, we assign a number to each event $E_{ij}$ \[x(E_{ij})= \frac{C}{|G|^{\frac{\delta}{2}} }\]
For large enough $|G|$, that assures that $\forall i,j$
\begin{equation}\label{eqa}\tag{\ding{56}} \Pr(E_{ij})\le C|G|^{-\delta} \le x(E_{ij}) \prod_{B\in \Gamma(E_{ij})} (1-x(B))\end{equation} 

We will see why the second inequality is true. Let's define
\[ Y:=\prod_{B\in \Gamma(E_{ij})} (1-x(B))= \left[1-\frac{C}{|G|^{\frac{\delta}{2}}} \right]^{2(K-1)} \ge 1- \frac{2(K-1)C}{|G|^{\frac{\delta}{2}}}\]
We set \[N={ ( 4(K-1)C ) }^{\frac{2}{\delta}} \]  and we assume $|G| \ge N$.

$N$ depends only on the rank and the choice of $\mathcal{A}$.

We get \[Y\ge \frac{1}{2}\]

Equation  \ref{eqa}  is satisfied if $|G|^{\frac{\delta}{2}} > 2$, which is indeed the case.

According to Lovász local lemma \autocite[lemma 5.1.1 on pg.64]{alon2004probabilistic}, if $E_{ij}$ are mutually independent of all the events that are not its neighbors , and there is an assignment $x(E_{ij})$ in [0,1] 
s.t. equation \ref{eqa} is satisfied, then \[\Pr\left(\bigcap \bar{E_{ij}}\right) \ge \prod_{ij} (1-x(E_{ij}))\]
\[= \left(1- \frac{C}{|G|^{ \frac{\delta}{2}}} \right)^{K\choose 2} \ge 1- \frac{C}{|G|^{ \frac{\delta}{2}}} {K\choose 2} \]

\end{proof}

We can now describe the expansion properties of a randomly generated HDZ:

\begin{prop}\label{liethm_B}
Let $\mathcal{A}$ be a complex that is $d$-regular,  $\chi$-strongly-colorable, s.t. for every projection $P^k,c,c'$  \[|P^k(E^1_{ cc' }(\mathcal{A}))|\ge 2\] 
where $E^1_{c c'}$ are the edges between colors $c$ and $c'$ ,  $P^k$ in a projection into $V^k$ ($k\in \{ c,c' \}$)  .

Let $C: V(\mathcal{A}) \to [\chi]$ be a coloring of $\mathcal{A}$. $K_c$ is the number of vertices of color $c$ . 
Suppose $\mathcal{A}$ satisfies property $\Inv$. 

   \medskip \medskip Let $\mathcal{G}  = G_1 \times G_2 \ldots \times G_\chi $ where $G_i$ are product of at most $r$ finite simple (or quasisimple) groups of Lie type of rank at most $r$. Additionally, no simple factor of $G_i$ is isomorphic to a simple factor of $G_j$ for $i\neq j$.	
We assume $\forall i\ |G_i|\ge N $ where $N$  depends only on the ranks of $G_i$ and the choice of $\mathcal{A}$.

  \medskip \medskip   We randomize $F_1,F_2 \ldots , F_\chi$ subsets of the corresponding groups of corresponding sizes $K_1 , K_2 , \ldots K_\chi$ uniformly independently.
We order them such that $\widetilde{\Inv}$ is satisfied.

We let $\mathcal{C}$ be
\[HDZ^{+}[\mathcal{A} ,C, \mathcal{G}, {F_1,F_2, \ldots , F_\chi} ]\] 
Then, 
in probability at least $1-O(|G_i|^{\delta})$ where $G_i$ is the smallest component of $G$,  random walk on $\mathcal{C}$ converges with rate $\lambda$ where $\lambda,\delta$  depend only on ranks of $G_i$ and the choice of $\mathcal{A}$.

\end{prop}

\begin{proof}
We want to reduce the condition on $M_{cc'}$ to condition on the corresponding 
projection of the edges on $G_c$ and $G_c'$.  That is on \[M^k_{ cc' }:=Cay(G_c,P_k(S_{cc'}))\]  
To do so,we use the following proposition:
\begin{displayquote}
\begin{prop}
    \label{proplie}
    (\cite[ Proposition 8.4. ]{breuillard2013expansion})
    let \(r \in \mathrm{N}\) and \(\epsilon>0 .\) suppose \(G=G_{1} G_{2},\) where \(G_{1}\) and \(G_{2}\) are products of at most \(r\) finite simple (or quasisimple) groups of Lie type of rank at most
r. Suppose that no simple factor of \(G_{1}\) is isomorphic to a simple factor of \(G_{2} .\) If \(x_{1}=\) \(x_{1}^{(1)} x_{1}^{(2)}, \ldots, x_{k}=x_{k}^{(1)} x_{k}^{(2)}\) are chosen so that \(\left\{x_{1}^{(1)}, \ldots, x_{k}^{(1)}\right\}\) and \(\left\{x_{1}^{(2)}, \ldots, x_{k}^{(2)}\right\}\) are
both \(\epsilon\) -expanding generating subsets in \(G_{1}\) and \(G_{2}\) respectively, then \(\left\{x_{1}, \ldots, x_{k}\right\}\) is \(\delta\)-expanding in G for some \(\delta=\delta(\epsilon, r)>0\)
\end{prop}
\end{displayquote}
\medskip \medskip     Lets assume that for every $c,c'$, and for every $k\in\{ c,c' \}$, $M^k_{cc'}$ is $\epsilon$-expander. 

    Then exists $\delta(rk(G_c),rk(G_c'),\epsilon)$ s.t. $M_{cc'}$ is a $\delta$-expander for every $c,c'$. We can take the minimum $\delta$, 
    and get that condition ~\eqref{eq:einstein} is satisfied.

    Next, we rely on randomization properties of Lie groups in order to assure that $M^k_{cc'}$ are all $\epsilon$-expanders.

    We use  \cref{lem:random} for every color $c$ separately, with the corresponding $G_c, K_c$ and get $\epsilon_c$-expansion. 
The probabilities are independent.
For every projection $P^k,c,c'$  \[|P^k(E^1_{ cc' }(\mathcal{A}))|\ge 2\]

That means that $M^k_{cc'}$ generated by at least two elements. 
So we have that in probability at least $\prod_{c\in [\chi]} (1- \frac{C(G_c)}{|G_c|^{\delta'(G_c)}} {K_c\choose 2}) $, every $M^k_{cc'}$ is a $\tilde{\epsilon}$ expander, for $\tilde{\epsilon}=min(\epsilon_c)$. 

If we set $\delta''$ to be the minimal $\delta'$, then in probability at least $1-O(|G_i|^{\delta''})$ where $G_i$ is the smallest component, condition \ref{eq:einstein} is satisfied, and we can use  \cref{genthmA} as $\mathcal{S}$ is defined over the correct vertices.
Therefore, the complex has a convergence rate of at least $\lambda$, depending only on ranks and the choice of $\mathcal{A}$.
\end{proof}

\section{Main theorem}
We now turn to proof the main theorem, for which need to combine the expansion with the symmetry and the links properties. 

\liethm

\begin{proof}
    We have $\mathcal{A}$ that is $d$-regular and $\chi$-strongly colorable with a coloring\footnote{the additional required conditions are satisfied too} $C$ .

Let $F_1,F_2 \ldots , F_\chi$ symmetric subsets of the corresponding groups of corresponding sizes $2K_1, \ldots 2K_\chi$ chosen uniformly independently.

    We have $\mathcal{A}' = HPOWER[\mathcal{A}]$  with corresponding coloring function $C'$. 

    By \cref{claim:HPOWER},  $\mathcal{A}'$ is $2d$-regular, $\chi$-colorable and satisfies property $Inv$ 
    (the  other regularity properties of $\mathcal{A}'$ are also obtained , i.e.  $\mathcal{A}'(2)=8\mathcal{A}(2)$ ) 

    We use  \cref{liethm_B}, to get the required expansion in the required probablity.

    The regularity properties are obtained from \cref{lem:properties} as  $\mathcal{C}=HDZ^{-}[\mathcal{A}',C',F_1 \ldots F_\chi]$ 

 We let 
    \[\mathcal{S}=CONV[\mathcal{A}',C,F_1,F_2,\ldots,F_\chi]\]
our complex is 
\[ \mathcal{C}=Sc[\mathcal{S},\mathcal{G}] \]
We can use this fact in \cref{lem:symmetry} to get that $\mathcal{C}$ is transitive.

       Then, finally, we use the lemma about the links to get the link properties (\cref{lemlink}).
\end{proof}

\pagebreak
\part{Additional Applications}
\label{partc}
Here we show how the method provide better convergence rate for two known constructions.
\section{The Construction By Conlon}
\label{ssub:conlon}
We can see that a special case of this construction is the construction by Conlon \cite{conlon2017hypergraph}.
This would be an illustrative example.

Conlon looked at $Cay(G,S)$ where $S$ is a set of generators with no non-trivial 4-cycles, and $G=\mathbb{F}^n_2$. He built a complex $\mathcal{C}$ which is based upon triangles of this graph. The triples of $\mathcal{C}$ are composed of 3 different vertices adjacent to the same vertex, and were divided to cliques naturally.

\medskip  \medskip We have a set $S\subset G$ s.t. $S=S^{-1}$.\\ We define the Conlon's complex $\mathcal{C}$ by its triangles:

\[
    \mathcal{C}(2)=\{s_{a}g,s_{b}g,s_{c}g\ |s_{a},s_{b},s_{c}\in S\ distinct\ and\ g\in G\}
\]

In our case it is enough to define $\mathcal{S}={S \choose 3}$
and $\mathcal{C}=CTS[ G,\mathcal{S} ]$. We require that there are no non-trivial 4-cycles in $Cay(G,S)$
for {\em Condition \hyperref[item4]{\textbf{D}}} to be satisfied. Indeed, this is equivalent, because a non-trivial $4$-cycle in $Cay(G,S)$ is
\[
	abcd=e
\]
 \[a,b,c,d\in\ S\text{ s.t. }\{c,d\}\neq\{a,b\}\]
 Which contradicts {\em Condition \hyperref[item4]{\textbf{D}}}. \\

So far we have described a small generalization of Conlon's construction. To describe it specifically, we require that $G=\mathbb{F}_{2}^{t}$ and the product
is additive.\\

\subsection{The random walk in Conlon case}

\label{sec:randcon}
We describe the random walk in Conlon's case, in terms of types and centers. Suppose we start
at $\{s_{1}g,s_{2}g\}$ or $E(g,\{s_{1},s_{2}\})$. In each step, we pick first a center
our edge is contained in. That is, we choose $k\in\{1,2\}$. So we will be 
either in center $g$ or in center $s_{1}s_{2}g$ in probability $1/2$. 
Now we choose another $s\in S$, where $s\neq s_{1},s_{2}$.

Then we look at the triangle that contains the edge $\{s_{1}g,s_{2}g,sg\}$
for $s\in S$. The type of the new edge is $\{s,s_{j}\}$ for $j\in\{1,2\}$.

So the choices are exactly $2d$, where $d$ is the degree of $L$.
The graph $L$ is exactly $J(S,2)$ defined earlier, so $d$ is $2(S-2)$. 
There are $4(S-2)$ choices all in all.

Looking it as a random walk over $G_{rep}$, selecting a center corresponds to selecting either the first operand or second operand in $T$. 
And selecting the new edge type corresponds to an action by $P_R$.

\subsection{Result}
We have the following result:

\begin{restatable}{cor}{conlons}\label{conlon}
	Assuming
	\[G=\mathbb{F}_{2}^{t}\]
	\[S\subset\ G\text{ such that }\]
	\[a+b=c+d\text{
	    only in trivial case }(\{a,b,c,d\}\subset\ S)\]

			Then for \[\mathcal{S}={S \choose 3}\text{ }H=CTS[ G,\mathcal{S} ]\]

	in terms of the original graph \[\lambda=\lambda(Cay(G,S))\]
	we have that the convergence rate on $2D$-random walk is \[
	 \frac{\sqrt{3}}{2} +\frac{1}{2\sqrt{3}}\lambda^2 + O(\frac{1}{d})
	\] 
\end{restatable}
\begin{proof}
    Directly from \ref{corollary:rw} for the defined CTS
\end{proof}

\begin{rem*}
	In terms of expansion of the auxiliary graph,
	we get as $\varepsilon\rightarrow 1$, asymptotic behavior of 
\[
	1 - \frac{\sqrt{3}}{2} - \frac{(\varepsilon - 1)^2}{2 \sqrt{3}} 
\]

	compared to 

	\[
		\frac{\epsilon^{4}}{2^{15}}
	\]

	achieved in \cite{conlon2017hypergraph}. 
	This is asymptotically better\footnote{It is possible that link analysis would yield even better results }.
\end{rem*}

\section{The 3-product case}
\label{three-product_case}
Chapman, Linal and Peled described a construction called Polygraph in the paper \cite{chapman2018expander}. We will describe it very briefly, and refer the reader to the paper for further explanation.

In this construction, one takes a graph $G$ with large enough girth and a multiset of numbers $S$.
And one defines a graph $G_S$ called polygraph. 
\par The vertices of $G_S$ are ${V(G)}^m$ (tensor product) 

Two vertices $(x_1\ldots x_n),(y_1\ldots y_n)$ are adjacent if the collection $(d(x_i,y_i)\ |\ 1\le i\le m )$ is equal as a multiset to $S$,
where $d$ is the distance function on the graph. 

Finally, one takes the cliques complex of this complex $\mathcal{C}_{G_{S}^{(2)}}$.

In the HDZ construction, If we take $\mathcal{G}=G_1\times G_2 \times G_3 $ with $S_1,S_2,S_3$ as generators and $\mathcal{A}=K^3_3$ (the complete 3-partite graph), we get a very similar construction as the [1,1,0] construction in the paper.

Namely, we get a complex $\mathcal{C}=HDZ(\mathcal{G},\mathcal{A},C,F_1,F_2,F_3)$
with triangles
	\[
		T=\{s_{1}g_{1},s_{2}g_{2},s_{3}g_{3}\ |\ g_{i}\in G_{i},s_{i}\in S_{i}\}
	\]

	\medskip \medskip  For $S=[1,1,0]$, $\mathcal{C}_{G_{S}^{(2)}}$ is the same as the complex $\mathcal{C}$, in the specific case $G_1=G_2=G_3$.
So, we provide a slight generalization of the $[1,1,0]$ case, as we allow taking different base graphs. 
On the other hand, we force all the graphs to be Cayley graph.

\begin{defn}[3-Product-Case]\label{def:3-product-case} 

	Given $G_1,G_2,G_3$ groups,\\ with $S_{i}\subset G_{i}$
	s.t.
	\begin{enumerate}
		\item $S_{i}=S_{i}^{-1}$
		\item $d:=|S_{i}|=|S_{j}|$
		\item $|G_{i}|=|G_{j}|$
		 
	\end{enumerate}
	We define
	\[
		\mathcal{G}:=G_{1}\times G_{2}\times G_{3}
	\]

	We can also describe it as a HDZ. 
	by defining $\mathcal{S}$ to be the 2-complex with faces \footnote{That means $\{ s_1,s_2,s_3 \}$ where $s_i\in S_i$ } 
	\[\mathcal{S}(2)=\{S_{1},S_{2},S_{3}\}\]
	
	as $HDZ[\mathcal{S},C,\mathcal{G} , S_1,S_2,S_3 ]$ with the obvious coloring (vertex $S_i$ is in color $i$)

	and \[
		H:=CTS[ G,\mathcal{S} ]
	\]
\end{defn}

And we have the following corollary:

\begin{restatable}{cor}{thrprod}
	Given $\lambda_{i}=\lambda(Cay(G_{i},S_{i}))$ ordered s.t. $\lambda_{1}\le\lambda_{2}\le\lambda_{3}$
	the random walk on $H$ defined in \defref{3-product-case}, converges
	with rate $\sqrt{\frac{1}{2}+\frac{1}{2}f(\frac{1+2\lambda_{3}}{3},\frac{1}{2})}$
	where $f$ is the zig-zag function\footnote{originally defined in \cite[Theorem 3.2]{reingold2002entropy}} from \cref{equation:zig-zag} \\ ( $f(a,b)\le a+b$, $f<1$ where
 $a,b<1$ ). 
\end{restatable}
\begin{proof}
    Directly from \cref{fullskel2}
\end{proof}

   \printbibliography
\end{document}